\documentclass[final]{siamart190516}
\usepackage{amsmath}
\usepackage{amsfonts}
\usepackage{algorithm}
\usepackage{algorithmic}
\usepackage{array}
\usepackage{url}
\usepackage{hyperref}
\usepackage{url}
\usepackage{bm}
\usepackage{graphicx}
\usepackage{verbatim}
\usepackage{placeins}
\usepackage{multicol}
\usepackage{capt-of}
\usepackage{booktabs}

\DeclareMathOperator*{\argmin}{arg\,min}

\DeclareMathOperator{\diag}{diag}

\newcommand{\grad}[1]{\nabla #1}
\newcommand{\tcol}[1]{\multicolumn{2}{c|}{#1}}
\newcommand{\defq}{\overset{\operatorname{def}}{=}}
\newcommand{\conv}[1]{\operatorname{conv}(#1)}
\newcommand{\eabs}{\epsilon_{\operatorname{abs}}}
\newcommand{\erel}{\epsilon_{\operatorname{rel}}}

\newcommand{\TheTitle}{An optimal lower bound for smooth convex functions}
\newcommand{\TheAuthors}{Mihai I. Florea and Yurii Nesterov}
\newcommand{\TheKeywords}{smoothing, bundle, memory, acceleration, adaptive, primal-dual, optimized gradient method}

\hypersetup{
colorlinks = false,
frenchlinks = false,
pdfborder = {0 0 0},
naturalnames = false,
hypertexnames = false,
breaklinks,
colorlinks = true,
allcolors = {green!50!black},
urlcolor = {red!50!black},
pdfinfo={
Title={\TheTitle},
Author={\TheAuthors},
Subject={\TheTitle},
Keywords={\TheKeywords}
}
}

\title{\TheTitle}

\author{Mihai I. Florea\thanks{Department of Mathematical Engineering (INMA), Catholic University of Louvain (UCL), Belgium; Department of Electronics and Computers, Transilvania University of Bra\c{s}ov, Romania. E-mail: \email{mihai.florea@uclouvain.be}}. \and Yurii Nesterov\thanks{Center for Operations Research and Econometrics (CORE), UCL. E-mail: \email{yurii.nesterov@uclouvain.be}}.
This project has received funding from the European Research Council (ERC) under the European Union's Horizon 2020 research and innovation programme (grant agreement No. 788368).}

\headers{\TheTitle}{\TheAuthors}

\begin{document}

\maketitle

\begin{abstract}
First order methods endowed with global convergence guarantees operate using global lower bounds on the objective. The tightening of the bounds has been shown to increase both the theoretical guarantees and the practical performance. In this work, we define a global lower bound for smooth differentiable objectives that is optimal with respect to the collected oracle information. The bound can be readily employed by the Gradient Method with Memory to improve its performance. Further using the machinery underlying the optimal bounds, we introduce a modified version of the estimate sequence that we use to construct an Optimized Gradient Method with Memory possessing the best known convergence guarantees for its class of algorithms, even in terms of the proportionality constant. We additionally equip the method with an adaptive convergence guarantee adjustment procedure that is an effective replacement for line-search. Simulation results on synthetic but otherwise difficult smooth problems validate the theoretical properties of the bound and proposed methods.
\end{abstract}

\section{Introduction}

The global minimum of an objective function can be reliably approximated only if the function exhibits some global property. One such property is convexity and it can be defined as the existence at every point of a supporting hyperplane on the entire function graph. This global lower bound on the objective is determined only by the first-order oracle information at that point: the gradient and the function value. For smooth objectives, the Gradient Method (GM) queries the oracle at the current iterate and constructs the corresponding supporting hyperplane using this information. The estimate function manages to accelerate GM by incorporating an accumulated hyperplane lower bound that is generally tighter than the GM lower bound at the optimum. Further building on this concept, the Gradient Methods with Memory (GMM) \cite{ref_021,ref_007,ref_008} construct a piece-wise linear lower bound as the maximum between all hyperplanes generated by the oracle information stored in memory. When the GMMs store all past information, the piece-wise linear bound is tighter than any weighted average of the constituent pieces, such as the hyperplane contained in the estimate function. This improvement raises the question of what is the tightest lower bound on a smooth objective function that we can construct based on the available oracle information.

Our answer consists in an interpolating global lower bound with provable optimality. This bound constitutes a new object in convex analysis, exhibiting both primal and dual characteristics. We elaborate on its remarkable properties.

We show how the memory footprint of our bound can be reduced while preserving its basic properties. This reduced bound is compatible with the mechanics of GMM and we use it to construct an Improved Gradient Method with Memory (IGMM) employing a tighter lower model than the bundle of GMM. However, the increased accuracy of the model does not translate into an increased worst-case rate.

Next, we study how our bound can be employed by accelerated schemes, with particular focus on how it can lead to improved theoretical convergence guarantees. For instance, the previously introduced Efficient Accelerated Gradient Method with Memory (AGMM)~\cite{ref_006,ref_008} takes advantage of the increased tightness of a piece-wise linear lower bound by dynamically adjusting the convergence guarantees. While the guarantees of AGMM are improved at runtime, its model does not allow significantly higher worst-case rates.

The work in \cite{ref_004, ref_012, ref_002} has shown that the proportionality constant in the worst-case rate of the Fast Gradient Method~\cite{ref_017} can be improved by a factor of 2. This may be the highest level available to a black-box method, as argued in \cite{ref_013} and partially supported by \cite{ref_003}[Theorem 3]. The analysis in \cite{ref_002} involves potential (Lyapunov) functions and closely resembles the pattern introduced in \cite{ref_005,ref_009,ref_010} while additionally utilizing the fact that the gradient of a smooth unconstrained objective at the optimum is zero. The analysis relies on the Performance Estimation Framework~\cite{ref_004} to provide the update rules for certain quantities present in the potential functions. However, these do not appear to have a definitive meaning \cite{ref_002}. Moreover, the lack of an estimate function formulation precludes the use of a convergence guarantee adjustment procedure at runtime, such as the ones proposed in \cite{ref_006,ref_008,ref_015}. More generally, the opaque nature of the potential function quantities hinders the development of any adaptive mechanism.

Building on the structure of the optimal bounds, we propose the framework of the primal-dual estimate functions (PDEF), a generalization of the original FGM estimate functions described in \cite{ref_018}. The PDEFs allow the creation of an Optimized Gradient Method with Memory (OGMM) with the worst-case rate also increased to the highest known level for its class of algorithms. The estimate function updates are straightforward and the estimate function optima allow for an adaptive increase of the convergence guarantees at runtime, beyond the worst-case ones. Augmentation, as proposed in \cite{ref_010}, leads to the potential functions described in \cite{ref_002}, in the process explaining the mechanism underlying the update rules and the meaning of the constituent quantities.

Preliminary simulation results on synthetic but difficult smooth problems confirm the superiority of our bound and of the methods employing it, either directly as IGMM or in primal-dual form as OGMM.

\subsection{Problem setup and notation}

We operate over the $n$-dimensional real vector space $\mathbb{E}$. We denote by $\mathbb{E}^*$ its dual space, the space of linear functions $s$ of value $\langle s, x \rangle$ for all $x \in \mathbb{E}$. The space $\mathbb{E}$ is endowed with the Euclidean norm $\| . \|$ defined as
\begin{equation}
\| x \| = \sqrt{\langle B x, x \rangle }, \quad x \in \mathbb{E},
\end{equation}
where the symmetric positive definite linear operator $B$ maps $\mathbb{E}$ to $\mathbb{E}^*$.
The dual of this norm, denoted as $\| . \|_*$, is consequently given by
\begin{equation}
\| g \|_* = \sqrt{\langle g, B^{-1} g \rangle }, \quad g \in \mathbb{E}^*.
\end{equation}

We seek to solve the following unconstrained smooth optimization problem:
\begin{equation} \label{label_001}
\min_{x \in \mathbb{E}} f(x) .
\end{equation}
Here, $f$ is differentiable over the entire $\mathbb{E}$ and convex. The gradient of $f$, $\grad{f}$, is Lipschitz continuous with Lipschitz constant $L_f > 0$. We further assume that the optimization problem in \eqref{label_001} contains a non-empty set of solutions $X^*$.

We consider that an optimization method has queried and stored in memory at a certain stage all oracle information pertaining to $m$ points $z_i \in \mathbb{E}$, $i = 1, ..., m$. We denote by $\mathcal{Z}$ the set of all points $z_i$ and by $\mathcal{I}_{\mathcal{Z}}$ the set of all oracle information arising from $\mathcal{Z}$, given by $\mathcal{I}_{\mathcal{Z}} \defq \{ (z_i, f_i, g_i) \ | \ i = 1, ..., m \}$, where $f_i = f(z_i)$ and $g_i = \grad{f}(z_i)$ for all $i = 1, ..., m$. We assume that we do not know anything else about function $f$.

All the above information can be used to narrow down the class of functions $f$ belongs to. We denote this restricted class by $\mathcal{S}_{\mathbb{E}, L_f}(\mathcal{I}_{\mathcal{Z}})$, where $\hat{f} \in \mathcal{S}_{\mathbb{E}, L_f}(\mathcal{I}_{\mathcal{Z}})$ if and only if all of the following hold:
\begin{enumerate}
\item $\hat{f}$ is convex and differentiable on $\mathbb{E}$;
\item $\grad{\hat{f}}$ is Lipschitz continuous with constant $L_f$;
\item $\hat{f}(z_i) = f_i$ and $\grad{\hat{f}}(z_i) = g_i$ for all $i = 1, ..., m$.
\end{enumerate}

\section{An interpolating lower bound} \label{label_002}

With the function class defined, it remains to determine its smallest member, if it exists. We start by constructing simple lower bounds, gradually increasing their complexity until we obtain a clear answer.

\subsection{Primal-dual global bounds} \label{label_003}

We consider an arbitrary $\hat{f} \in \mathcal{S}_{\mathbb{E}, L_f}(\mathcal{I}_{\mathcal{Z}})$. The structure of $\hat{f}$ allows us to define a set of global bounds that are independent of $\hat{f}$ itself.
First, each supporting hyperplane $l_i$ at point $z_i$ and the combined piece-wise linear \emph{lower} model $l$ are, respectively, given by
\begin{align}
l_i(y) &\defq f_i + \langle g_i, y - z_i \rangle, \quad y \in \mathbb{E}, \quad i = 1, ..., m, \\
l(y) &\defq \displaystyle \max_{i = 1, ..., m} l_i(y), \quad y \in \mathbb{E} .
\end{align}
The convexity of $\hat{f}$ implies that
\begin{equation} \label{label_004}
\hat{f}(y) \geq l(y) \geq l_i(y), \quad y \in \mathbb{E}, \quad i = 1, ..., m .
\end{equation}
The Lipschitz property of $\grad{\hat{f}}$ yields parabolic\footnote{We define parabolae as quadratic functions whose Hessian is a positive multiple of $B$.} \emph{upper} bounds at $z_i$, denoted as
\begin{equation} \label{label_005}
\Psi_i(y) \defq l_i(y) + \frac{L_f}{2} \| y - z_i \|^2 \geq \hat{f}(y), \quad y \in \mathbb{E}, \quad i = 1, ..., m .
\end{equation}
To obtain bounds that are tighter than the aforementioned ones, which happen to operate only in the primal space $\mathbb{E}$, we consider bounds containing both primal and dual information. We thus define functions $\phi_i$ and $\phi$ as
\begin{align}
\phi_i(y, g) &\defq l_i(y) + \frac{1}{2 L_f} \| g - g_i \|_*^2, \quad y \in \mathbb{E}, \quad g \in \mathbb{E}^*, \quad i = 1, ..., m, \label{label_006} \\
\phi(y, g) &\defq \max_{i = 1, ..., m} \phi_i(y, g), \quad y \in \mathbb{E}, \quad g \in \mathbb{E}^*. \label{label_007}
\end{align}
\begin{proposition} \label{label_008}
Function $\phi$ is convex in $y$, strongly convex in $g$. It also dominates the lower model $l$ for any value of $g$, namely
\begin{equation} \nonumber
\phi(y, g) \geq l(y), \quad y \in \mathbb{E}, \quad g \in \mathbb{E}^*.
\end{equation}
\end{proposition}
\begin{proof}
Function $\phi$ is the maximum of functions $\phi_i$ that are convex in $y$ and strongly convex in $g$.
From \eqref{label_006} we have that $\phi_i(y, g) \geq l_i(y)$, $i = 1, ..., m$, $y \in \mathbb{E}$ and $g \in \mathbb{E}^*$. Taking the maximum over all $i$ gives our result.
\end{proof}

\begin{proposition} \label{label_009}
Function $\phi$ can be used to construct a global lower bound on any function $\hat{f} \in \mathcal{S}_{\mathbb{E}, L_f}(\mathcal{I}_{\mathcal{Z}})$ as
\begin{equation} \nonumber
\hat{f}(y) \geq \phi(y, \grad{\hat{f}}(y)), \quad y \in \mathbb{E}.
\end{equation}
\end{proposition}
\begin{proof}
Adding and subtracting terms in \eqref{label_005} using an arbitrary $z \in \mathbb{E}$ yields
\begin{equation} \label{label_010}
\hat{f}(x) - \langle \grad{\hat{f}}(z), x \rangle \leq \hat{f}(y) - \langle \grad{\hat{f}}(z), y \rangle + \langle \grad{\hat{f}}(y) - \grad{\hat{f}}(z), x - y \rangle + \frac{L_f}{2} \| x - y \|^2,
\end{equation}
for all $x, y \in \mathbb{E}$. We define around $z$ the function $\theta_z(x) = \hat{f}(x) - \langle \grad{\hat{f}}(z), x \rangle$, $x \in \mathbb{E}$. With this notation, \eqref{label_010} becomes the $L_f$ smoothness condition for $\theta_z$, written as
\begin{equation}
\theta_z(x) \leq \theta_z(y) + \langle \grad{\theta_z}(y), x - y \rangle + \frac{L_f}{2} \| x - y \|^2, \quad x, y \in \mathbb{E}.
\end{equation}
The first-order optimality condition for $\theta_z$ implies that the global optimum of $\theta_z$ is attained at $x^* = z$ and we have that
\begin{align}
\theta_z(z) &= \min_{x \in \mathbb{E}} \theta_z(x) \leq \min_{x \in \mathbb{E}} \left\{\theta_z(y) + \langle \grad{\theta_z}(y), x - y \rangle + \frac{L_f}{2} \| x - y \|^2 \right\} \nonumber \\
&= \theta_z(y) - \frac{1}{2 L_f} \| \grad{\theta_z}(y) \|_*^2, \quad y \in \mathbb{E}. \label{label_011}
\end{align}
Considering that $z$ is arbitrary, setting it to $z_i$ and expanding the $\theta_z$ terms accordingly in \eqref{label_011} yields
\begin{equation}
\hat{f}(y) \geq \phi_i(y, \grad{\hat{f}}(y)), \quad y \in \mathbb{E}, \quad i = 1, ..., m ,
\end{equation}
which completes the proof. The above results build on \cite[Theorem 2.1.15]{ref_018}.
\end{proof}

\subsection{Definition and oracle functions} \label{label_012}

Whereas $l$ was the lower bound employed by the original Gradient Methods with Memory~\cite{ref_021}, the primal-dual object in \eqref{label_007} along with Proposition~\ref{label_009} suggest that we can construct a smooth function $\hat{f}_{\operatorname{min}} \in \mathcal{S}_{\mathbb{E}, L_f}(\mathcal{I}_{\mathcal{Z}})$, hence larger than $l$, that lower bounds \emph{any} $\hat{f} \in \mathcal{S}_{\mathbb{E}, L_f}(\mathcal{I}_{\mathcal{Z}})$.

Let function $p$ be defined by taking the minimum over the variable $g$ in \eqref{label_007} namely
\begin{equation} \label{label_013}
p(y) \defq \displaystyle \min_{g \in \mathbb{E}^*} \phi(y, g), \quad y \in \mathbb{E}.
\end{equation}
Proposition~\ref{label_008} implies that $p$ is well defined over the entire $\mathbb{E}$, is convex and has a unique optimal $g$ in \eqref{label_013} for every $y \in \mathbb{E}$. We denote this optimum as
\begin{equation} \label{label_014}
g^*(y) \defq \displaystyle \argmin_{g \in \mathbb{E}^*} \phi(y, g), \quad y \in \mathbb{E}.
\end{equation}
We can compute the values of $p(y)$ and $g^*(y)$ using the fact that
\begin{equation}
\phi(y, g) = \displaystyle \max_{\lambda \in \Delta_m} \displaystyle \sum_{i = 1}^{m} \lambda_i \left( f_i + \langle g_i, y - z_i \rangle + \frac{1}{2 L_f} \| g - g_i \|_*^2 \right) , \quad y \in \mathbb{E}, \quad g \in \mathbb{E}^*,
\end{equation}
where $\Delta_m$ is the standard simplex in $m$ dimensions. Let $f_*, d_{g} \in \mathbb{R}^m$ with
\begin{equation}
(f_*)_i \defq \langle g_i, z_i \rangle - f_i, \quad (d_{g})_i \defq \| g_i \|_*^2, \quad i = 1, ..., m,
\end{equation}
and $G \in \mathbb{E}^* \times \mathbb{R}^m$ with $G \defq (g_1, g_2, ..., g_{m})$.

Using the above notation we have that $p$ is given for all $y \in \mathbb{E}$ by
\begin{equation}
p(y) = \displaystyle \min_{g \in \mathbb{E}^*} \max_{\lambda \in \Delta_m} \langle \lambda, G^T y - f_* \rangle + \frac{1}{2 L_f} \|g\|_*^2 - \frac{1}{L_f} \langle \lambda, G^T B^{-1} g \rangle + \frac{1}{2 L_f} \langle \lambda, d_{g} \rangle.
\end{equation}
Strong duality holds in this case and we have for all $y \in \mathbb{E}$ that
\begin{equation} \label{label_015}
p(y) = \max_{\lambda \in \Delta_m} \left\langle \lambda, G^T y - f_* + \frac{1}{2 L_f} d_{g} \right\rangle + \frac{1}{L_f} \left( \min_{g \in \mathbb{E}^*} \frac{1}{2} \| g \|_*^2 - \langle G \lambda, B^{-1} g \rangle \right).
\end{equation}
The optimum for the inner minimization problem is reached when
\begin{equation} \label{label_016}
g = G \lambda .
\end{equation}
We can use \eqref{label_016} to eliminate variable $g$ from \eqref{label_015} and obtain
\begin{equation} \label{label_017}
p(y) = \max_{\lambda \in \Delta_m} \left\{ \rho(y, \lambda) \defq \left\langle \lambda, G^T y - f_* + \frac{1}{2 L_f} d_{g} \right\rangle - \frac{1}{2 L_f} \| G \lambda \|_*^2 \right\} , \quad y \in \mathbb{E}.
\end{equation}
Let the optimal point set of problem~\eqref{label_017} be given by
\begin{equation}
\Lambda^*(y) \defq \left\{ \lambda^* \in \mathbb{R}^m \ \middle| \ \rho(y, \lambda^*) = p(y) \right\} , \quad y \in \mathbb{E}.
\end{equation}

Based on the equivalence between the two expressions for $p$ in \eqref{label_013} and \eqref{label_017}, respectively, we state the following result.

\begin{proposition} \label{label_018}
Function $p$ is differentiable, with the gradient given by
\begin{equation} \label{label_019}
\grad{p}(y) = g^*(y) = G \lambda^*, \quad \lambda^* \in \Lambda^*(y) , \quad y \in \mathbb{E}.
\end{equation}
\end{proposition}
\begin{proof}
By applying Danskin's theorem~\cite{ref_001} in \eqref{label_017} we obtain an expression for the subdifferential of $p$ at $y$ in the form of
\begin{equation} \label{label_020}
\partial p(y) = \left\{ \frac{\partial \phi}{\partial y}(y, \lambda^*) \ \middle| \ \lambda^* \in \Lambda^*(y) \right\} = \left\{G \lambda^* \ \middle| \ \lambda^* \in \Lambda^*(y) \right\} , \quad y \in \mathbb{E}.
\end{equation}
Recall that we have performed in \eqref{label_017} the variable change in \eqref{label_016}. Therefore, $g^* = G \lambda^*$ is optimal in \eqref{label_013} for any $\lambda^*$ in $\Lambda^*(y)$. However, we have established that $g^*$ is unique and given by \eqref{label_014}. Taking also into account the subdifferential expression in \eqref{label_020}, we get the desired result.
\end{proof}

\subsection{Fundamental properties} First, we study how $p$ relates to $\hat{f}_{\operatorname{min}}$.

\begin{proposition} \label{label_021}
Function $p$ is convex and lies between $l$ and any $\hat{f} \in \mathcal{S}_{\mathbb{E}, L_f}(\mathcal{I}_{\mathcal{Z}})$, namely
\begin{equation} \nonumber
l(y) \leq p(y) \leq \hat{f}(y), \quad y \in \mathbb{E}.
\end{equation}
\end{proposition}
\begin{proof}
Proposition~\ref{label_008} gives the convexity of $p$ and the first inequality.
Likewise, from Proposition~\ref{label_009} we have that
\begin{equation}
\hat{f}(y) \geq \phi(y, \grad{\hat{f}}(y)) \geq \min_{g \in \mathbb{E}^*} \phi(y, g) = p(y) , \quad y \in \mathbb{E}.
\end{equation}
\end{proof}

\begin{proposition} \label{label_022}
The function $p$ has an $L_f$-Lipschitz continuous gradient.
\end{proposition}
\begin{proof}
Let $p_*: \mathbb{E}^* \rightarrow \mathbb{E}$ be the convex conjugate (Fenchel dual)~\cite{ref_023} of $p$, given by
\begin{equation}
p_*(s) = \displaystyle \max_{y \in \mathbb{E}} \{ \langle s, y \rangle - p(y) \}, \quad s \in \mathbb{E}^*.
\end{equation}
From the dual formulation of $p$ in \eqref{label_015} we have that strong duality holds and thus
\begin{equation} \label{label_023}
\begin{aligned}
p_*(s) &= \displaystyle \max_{y \in \mathbb{E}} \left\{ \langle s, y \rangle - \displaystyle \max_{\lambda \in \Delta_m} \phi(y, \lambda) \} \right\} = \displaystyle \min_{\lambda \in \Delta_m} \max_{y \in \mathbb{E}} \left\{ \langle s, y \rangle - \phi(y, \lambda) \right\} \\
& = \displaystyle \min_{\lambda \in \Delta_m} \left\{ \max_{y \in \mathbb{E}} \left\{ \langle s - G \lambda, y \rangle \right \} + \langle \bar{f}_*, \lambda \rangle + \frac{1}{2 L_f} \| G \lambda \|_*^2 \right \}, \quad s \in \mathbb{E}^*,
\end{aligned}
\end{equation}
where $\bar{f}_* \defq f_* - \frac{1}{2 L_f} d_g$. We also define $\conv{G} = \left\{ G\lambda \ \middle| \ \lambda \in \Delta_m \right\}$ as the convex hull of the columns of G and
$\Lambda_G(s) = \left\{ \lambda \in \Delta_m \ \middle| \ G \lambda = s \right\}$ for all $s \in \mathbb{E}^*$.
From \eqref{label_023} it follows that $p_*(s) = +\infty$ for all $s \notin \conv{G}$. Otherwise, we have that
\begin{equation} \label{label_024}
p_*(s) = \frac{1}{2 L_f} \| s \|_*^2 + \bar{p}_*(s), \quad s \in \conv{G},
\end{equation}
where
\begin{equation} \label{label_025}
\bar{p}_*(s) = \displaystyle \min_{\lambda \in \Delta_m \cap \Lambda_G(s)} \langle \bar{f}_*, \lambda \rangle, \quad s \in \conv{G}.
\end{equation}
Being defined on a bounded convex set, function $\bar{p}_*$ is real-valued and the optimization problem for each $s$ has a non-empty set of optimal multipliers denoted by $\Lambda_{\bar{p}_*}(s)$.
Next, we prove the convexity of the function $\bar{p}_*$ on $\conv{G}$. We consider $s_1$ and $s_2$ to be two arbitrary points in $\conv{G}$. Let $\lambda_1$ and $\lambda_2$ be arbitrary members of $\Lambda_{\bar{p}_*}(s_1)$ and $\Lambda_{\bar{p}_*}(s_2)$, respectively. We have for all $0 < \alpha < 1$ that
\begin{equation}
(1-\alpha) \bar{p}_*(s_1) + \alpha \bar{p}_*(s_2) = (1-\alpha) \langle \bar{f}_*, \lambda_1 \rangle + \alpha \langle \bar{f}_*, \lambda_2 \rangle = \langle \bar{f}_*, \lambda_3 \rangle,
\end{equation}
where $\lambda_3 = (1-\alpha) \lambda_1 + \alpha \lambda_2$. However, $\lambda_3$ belongs to $\Lambda_3$ where $\Lambda_3 = \Delta_m \cap \Lambda_G( (1-\alpha) s_1 + \alpha s_2 )$ and we therefore have
\begin{equation} \label{label_026}
(1-\alpha) \bar{p}_*(s_1) + \alpha \bar{p}_*(s_2) = \langle \bar{f}_*, \lambda_3 \rangle \geq \displaystyle \min_{\lambda \in \Lambda_3} \langle \bar{f}_*, \lambda \rangle = \bar{p}_*((1-\alpha) s_1 + \alpha s_2).
\end{equation}
The arbitrary nature of $s_1$, $s_2$, $\lambda_1$ and $\lambda_2$ in \eqref{label_026} establishes the convexity of $\bar{p}$ on $\conv{G}$ which together with \eqref{label_024} implies the $1 / L_f$ strong convexity of $p_*$ on $\mathbb{E}^*$. Because $p$ is a closed convex function it follows that $p_*$ is also closed convex. Since $p_*$ is proper, it is also subdifferentiable. Thus, the results in \cite{ref_019}[Theorem~1] extend to $p$ and imply that $p$ has a Lipschitz continuous gradient with constant $L_f$.
\end{proof}

\begin{proposition} \label{label_027}
The oracle information in $\mathcal{I}_{\mathcal{Z}}$ applies to $p$ as well, namely $p(z_i) = f_i$ and $\grad{p}(z_i) = g_i$ for all $i = 1, ..., m$.
\end{proposition}
\begin{proof}
Herein we take all $i = 1, ..., m$.
From Proposition~\ref{label_021} we have that $f_i = l(z_1) \leq p(z_i) \leq \hat{f}(z_i) = f_i$ for an arbitrary $\hat{f} \in \mathcal{S}_{\mathbb{E}, L_f}(\mathcal{I}_{\mathcal{Z}})$. Propositions \ref{label_021} and \ref{label_022} further imply that $g_i = \grad{l}(z_i) \leq \grad{p}(z_i) \leq \grad{\hat{f}}(z_i) = g_i$.
\end{proof}

Propositions \ref{label_018}, \ref{label_022} and \ref{label_027} certify that $p$ is a member of $\mathcal{S}_{\mathbb{E}, L_f}(\mathcal{I}_{\mathcal{Z}})$. Proposition~\ref{label_021} further states that $p$ is a lower bound on any function in $\mathcal{S}_{\mathbb{E}, L_f}(\mathcal{I}_{\mathcal{Z}})$. Thus, $p$ has all the properties of the sought after $\hat{f}_{\operatorname{min}}$. We conclude that $p = \hat{f}_{\operatorname{min}}$.

Our bound is also robust, in the sense that it exhibits a form of tilt associativity.
\begin{proposition} \label{label_028}
For all full-domain $L_f$-smooth and convex functions $f$, all linear functions given by $c \in \mathbb{E}^*$ and $d \in \mathbb{R}$, all collections of points $\mathcal{Z} =\{z_i \in \mathbb{E} \ | \ i = 1, ..., m\}$ with $m \geq 1$, and all $y \in \mathbb{E}$ we have that
\begin{equation} \nonumber
\mathcal{L}(f + \langle c, y \rangle + d, \mathcal{Z})(y) = \mathcal{L}(f, \mathcal{Z})(y) + \langle c, y \rangle + d,
\end{equation}
where the operator $\mathcal{L}$ produces the optimal lower bound $\mathcal{L}(\tilde{f}, \mathcal{Z})$ based on the oracle information set $\mathcal{I}_{\mathcal{Z}}(\tilde{f}) \defq \{ (z_i, \tilde{f}(z_i), \grad{\tilde{f}}(z_i)) \ | \ z_i \in \mathcal{Z} \}$ using the function $\tilde{f}$ that is also full-domain $L_f$-smooth and convex.
\end{proposition}
\begin{proof} Here, we consider all $y \in \mathbb{E}$ and $g \in \mathbb{E}^*$.

We define $\bar{f}(y) \defq f(y) + \langle c, y \rangle + d$ with oracle output $\bar{f}_i = \bar{f}(z_i)$ and $\bar{g}_i = \grad{\bar{f}}(z_i)$. The primal-dual bounds for $\bar{f}$ are thus given by
\begin{equation}
\bar{\phi}_i(y, g) \defq \bar{f}_i + \langle \bar{g}_i, y - z_i \rangle + \frac{1}{2 L_f} \| g - g_i\|_*^2.
\end{equation}
Taking into account that $\bar{f}_i = f_i + \langle c, z_i \rangle + d$ and $\bar{g}_i = g_i + c$ we have
\begin{equation} \label{label_029}
\begin{aligned}
\bar{\phi}_i(y, g) &= f_i + \langle c, z_i \rangle + d + \langle g_i + c, y - z_i\rangle + \frac{1}{2 L_f} \| g - g_i\|_*^2 \\
& = f_i + \langle g_i, y - z_i\rangle + \frac{1}{2 L_f} \| g - g_i\|_*^2 + \langle c, y \rangle + d = \phi_i(y, g) + \langle c, y \rangle + d.
\end{aligned}
\end{equation}
We conclude by taking the maximum over $i$ followed by the minimum over $g$ in \eqref{label_029}.
\end{proof}

\subsection{Information selection and aggregation} \label{label_030}
Storing the entire history of oracle calls in memory may not be practical in many contexts, particularly when the optimization method needs many iterations to produce results of adequate accuracy. In this section we study how the memory footprint of the model can be reduced while preserving its most useful properties.

\subsubsection{Primal bounds} \label{label_031}
One approach is to maintain the bundle employed by the Gradient Methods with Memory. The bundle addresses the memory limitations by storing a subset of the entire history, a weighted average or a mix thereof. Hence, we consider such an aggregated model of size $\tilde{m}$, not necessarily equal to $m$, storing $\tilde{f}_* \in \mathbb{R}^{\tilde{m}}$ and $\tilde{G} \in \mathbb{E}^* \times \mathbb{R}^m$. The derived model is obtained by applying the linear transform $\mathcal{T} \in \mathbb{R}^{m} \times \mathbb{R}^{\tilde{m}}$ to the history of size $m$ as:
$\tilde{f_*} = \mathcal{T}^T f_*$ and $\tilde{G} = G \mathcal{T}$. The mix between weighted averaging and selection equates to having all columns $\mathcal{T}_i$, $i = 1, ..., \tilde{m}$, within $\Delta_{m}$. The optimal lower bound $\tilde{p}$ produced by the derived model described by $\tilde{f}_*$ and $\tilde{G}$ is given by
\begin{equation} \label{label_032}
\tilde{p}(y) = \max_{\tilde{\lambda} \in \Delta_{\tilde{m}}} \left\{ \tilde{\rho}(y, \tilde{\lambda}) \defq \left\langle \tilde{\lambda}, \tilde{G}^T y - \tilde{f}_* + \frac{1}{2 L_f} \tilde{d}_g \right\rangle - \frac{1}{2 L_f} \| \tilde{G} \tilde{\lambda} \|^2_* \right\}, \quad y \in \mathbb{E},
\end{equation}
where $\tilde{d}_g \in \mathbb{R}^{\tilde{m}}$ is a vector that stores the squared dual norms for each column of $G$, namely $(\tilde{d}_g)_i = \| \tilde{G}_i \|^2_*$, $i = 1, ... , \tilde{m}$.
To prove our main result,
we need the following lemma.
\begin{lemma} \label{label_033}
For \emph{any} size $m \geq 1$ and any symmetric positive semidefinite $m \times m$ matrix $C$ and \emph{any} $\lambda \in \Delta_m$, we have that
\begin{equation}
\langle \lambda, C \lambda \rangle \leq \langle \lambda, \diag(C) \rangle,
\end{equation}
where $\diag(C)$ denotes the main diagonal of $C$.
\end{lemma}
\begin{proof}
Matrix $C$ admits a square root $R = C^{\frac{1}{2}}$ with columns $R_i$, $i = 1, ..., m$. We thus have that $C = R^T R$. Using $\sum_{i = 1}^{m} \lambda_i = 1$ we arrive at
\begin{equation} \label{label_034}
\begin{gathered}
\langle \lambda, \diag(C) \rangle - \langle \lambda, C \lambda \rangle
= \left( \displaystyle \sum_{j = 1}^{m} \lambda_j \right) \displaystyle \sum_{i = 1}^{m} \lambda_i \langle R_i, R_i \rangle - \displaystyle \sum_{i = 1}^{m} \displaystyle \sum_{j = 1}^{m} \lambda_i \lambda_j \langle R_i, R_j \rangle
\end{gathered}
\end{equation}
Multiplying both sides in \eqref{label_034} by two and rearranging terms we obtain
\begin{equation}
2 \left(\langle \lambda, \diag(C) \rangle - \langle \lambda, C \lambda \rangle\right) = \displaystyle \sum_{i = 1}^{m} \displaystyle \sum_{j = 1}^{m} \lambda_i \lambda_j \langle R_i - R_j, R_i - R_j \rangle \geq 0.
\end{equation}
\end{proof}

\begin{proposition} \label{label_035}
The optimal bound generated by the derived model using \eqref{label_032} is dominated by the bound using the entire oracle history \eqref{label_017}, namely
\begin{equation}
\tilde{p}(y) \leq p(y), \quad y \in \mathbb{E}.
\end{equation}
\end{proposition}
\begin{proof}
The objective in \eqref{label_032} can be rewritten using $C = G^T B^{-1} G$ as
\begin{equation} \label{label_036}
\tilde{\rho}(y, \tilde{\lambda}) = \langle \tilde{\lambda}, \mathcal{T}^T G^Ty - \mathcal{T}^T f_* \rangle + \frac{1}{2 L_f} \left( \langle \tilde{\lambda}, \diag(\mathcal{T}^T C \mathcal{T}) \rangle - \langle \tilde{\lambda}, \mathcal{T}^T C \mathcal{T} \tilde{\lambda} \rangle \right)
\end{equation}
By expanding terms, we have
\begin{equation} \label{label_037}
(\diag(\mathcal{T}^T C \mathcal{T}))_i = \langle \mathcal{T}_i, C \mathcal{T}_i \rangle, \quad (\mathcal{T}^T \diag(C))i = \langle \mathcal{T}_i, \diag(C) \rangle, \quad i = 1, ..., \tilde{m} .
\end{equation}
Applying Lemma~\ref{label_033} to \eqref{label_037} we obtain for all $\tilde{\lambda} \in \Delta_{\tilde{m}}$ that
\begin{equation} \label{label_038}
\langle \tilde{\lambda}, \diag(\mathcal{T}^T C \mathcal{T}) \rangle \leq \langle \tilde{\lambda}, \mathcal{T}^T \diag(C) \rangle = \langle \lambda, \diag(C) \rangle ,
\end{equation}
using the variable substitution $\lambda = \mathcal{T} \tilde{\lambda}$, with $\lambda$ spanning the set $\Lambda \defq \left\{ \mathcal{T} \alpha \ \middle| \ \alpha \in \Delta_{\tilde{m}} \right\}$. Combining \eqref{label_036} and \eqref{label_038} implies for all $\tilde{\lambda} \in \Delta_{\tilde{m}}$ and $y \in \mathbb{E}$ that
\begin{equation} \label{label_039}
\tilde{\rho}(y, \tilde{\lambda}) \leq \rho(y, \lambda) = \langle \lambda, G^T y - f_* \rangle + \frac{1}{2 L_f} \left( \langle \lambda, \diag(C) \rangle - \langle \lambda, C \lambda \rangle \right).
\end{equation}
Finally, observing that $\Lambda \subseteq \Delta_m$, we reach the desired result
\begin{equation}
\tilde{p}(y) = \max_{\tilde{\lambda} \in \Delta_{\tilde{m}}} \tilde{\rho}(y, \tilde{\lambda}) \overset{\eqref{label_039}}{\leq} \max_{\lambda \in \Lambda} \rho(y, \lambda) \leq \max_{\lambda \in \Delta_m} \rho(y, \lambda) = p(y), \quad y \in \mathbb{E}.
\end{equation}
\end{proof}
Proposition~\ref{label_035} implies that model selection and aggregation reduces memory usage at the expense of model tightness. However, $\tilde{p}$ remains a global lower bound on the objective, an essential property for any first-order scheme with global convergence guarantees. We also notice that constructing an aggregated model larger than the oracle history, corresponding to $\tilde{m} > m$, is detrimental in terms of memory requirements and does not improve bound tightness.

\subsubsection{Primal-dual bounds} \label{label_040} We also consider the aggregation of the primal-dual bounds in Subsection~\ref{label_003}. These bounds can be be rewritten as
\begin{align}
\phi_i(y, g) &= f_i + \langle g_i, y - z_i \rangle + \frac{1}{2 L_f} \| g - g_i \|_*^2 \label{label_041}\\
&= -(\bar{f}_*)_i + \left\langle g_i, y - \frac{1}{L_f} B^{-1} g \right\rangle + \frac{1}{2 L_f} \| g \|_*^2, \quad i = 1, ... , m, \label{label_042}
\end{align}
where $\bar{f}_* \defq f_* - \frac{1}{2 L_f} d_g$ carries over from the proof of Proposition~\ref{label_022}. Using \eqref{label_042} we define for all $\lambda \in \Delta_m$ the aggregate primal-dual bound
\begin{equation}
\Phi_{\lambda}(y, g) \defq \displaystyle\sum_{i = 1}^{m} \lambda_i \phi_i(y, g) = -\langle \lambda, \bar{f}_* \rangle + \left\langle G \lambda, y - \frac{1}{L_f} B^{-1} g \right\rangle + \frac{1}{2 L_f} \| g \|_*^2.
\end{equation}
Thus, for any $\lambda \in \Delta_m$ we have a global lower bound on $p$ in the form of
\begin{equation} \label{label_043}
\min_{g \in \mathbb{E}^*} \Phi_{\lambda}(y, g) = \rho(y, \lambda) \leq \max_{\lambda \in \Delta_m} \rho(y, \lambda) = p(y) , \quad y \in \mathbb{E}.
\end{equation}

\subsection{A geometric interpretation}

To better understand the shape of function $p$, we introduce a more intuitive third formulation, based entirely on geometric arguments. The main building blocks of our interpretation are parabolae defined earlier as simple quadratic functions whose Hessians are multiple of $B$. All parabolae can be written in the following canonical form:
\begin{equation} \label{label_044}
P_{\gamma, v, w}(y) \defq w + \frac{\gamma}{2} \| y - v \|^2 , \quad y \in \mathbb{E},
\end{equation}
where $\gamma > 0$, $v \in \mathbb{E}$ and $w \in \mathbb{R}$ represent the curvature, vertex (optimal point) and optimal value, respectively.

We consider the set of all parabolae with curvature $L_f$ that dominate $l$ (upper parabolae), written as
\begin{equation} \label{label_045}
\mathcal{P}(\mathcal{Z}) \defq \left\{ P_{L_f, v, w} \ \middle| \ v \in \mathbb{E}, \ w \in \mathbb{R}, \ P_{L_f, v, w}(y) \geq l(y), \ y \in \mathbb{E} \right\} .
\end{equation}
Let the lower envelope of all $\Psi \in \mathcal{P}(\mathcal{Z})$ be given by
\begin{equation} \label{label_046}
\hat{p}(y) \defq \displaystyle \min_{\Psi \in \mathcal{P}(\mathcal{Z})} \Psi(y), \quad y \in \mathbb{E}.
\end{equation}

\begin{proposition} \label{label_047}
The formulation in \eqref{label_046} is equivalent to \eqref{label_013}, namely
\begin{equation}
p(y) = \hat{p}(y), \quad y \in \mathbb{E} .
\end{equation}
\end{proposition}
\begin{proof}
We fix an arbitrary point $y \in \mathbb{E}$. We define parabola $\Psi_{p, y}(x)$ similarly to \eqref{label_005}, but this time based on $p$ as
\begin{equation} \label{label_048}
\Psi_{p, y}(x) \defq p(y) + \langle \grad{p}(y), x - y \rangle + \frac{L_f}{2} \| x - y \|^2, \quad x \in \mathbb{E}.
\end{equation}

First, Propositions \ref{label_021} and \ref{label_022} state that $\Psi_{p, y}(x) \geq p(x) \geq l(x)$ for all $x, y \in \mathbb{E}$. Therefore $\Psi_{p, y} \in \mathcal{P}(\mathcal{Z})$ for all $y \in \mathbb{E}$ and
\begin{equation} \label{label_049}
p(y) = \Psi_{p, y}(y) \geq \min_{\Psi \in \mathcal{P}(\mathcal{Z})} \Psi(y) = \hat{p}(y), \quad y \in \mathbb{E} .
\end{equation}
Next, we fix an arbitrary $\Psi = P_{L_f, v, w}$ in $\mathcal{P}(\mathcal{Z})$. We construct the function $c$ based on $\Psi$ as $c = \operatorname{conv}(\Psi, \Psi_{p, z_1}, \Psi_{p, z_2}, ..., \Psi_{p, z_m})$ where $\operatorname{conv}$ in this context denotes the function whose epigraph is given by the convex hull of the argument function epigraphs.

Using the expression in \eqref{label_044}, we can write $\Psi_{p, z_i}$ for $i = 1, ..., m$ in canonical form as $\Psi_{p, z_i} = P_{L_f, v_i, w_i}$, $i = 1, ..., m$, where $v_i \defq z_i - \frac{1}{L_f} B^{-1} g_i$ and $w_i \defq f_i - \frac{1}{2 L_f} \| g_i \|_* ^2$. We set $v_{m + 1} = v$ and $w_{m + 1} = w$. Using the fact that $c$ is proper convex along with \cite[Theorem 16.5]{ref_023} we obtain an expression for the Fenchel dual of $c$ as
\begin{align}
c_*(s) &= \max_{i = 1, ..., m + 1} \left(w_i + \frac{L_f}{2} \|y - v_i\|^2 \right)_* = \max_{i = 1, ..., m + 1} \left\{ \frac{1}{2 L_f} \| s \|_*^2 - w_i + \langle v_i, s \rangle \right\} \nonumber \\
&= \frac{1}{2 L_f} \| s \|_*^2 + \displaystyle \max_{i = 1, ..., m + 1} \left\{ \langle v_i, s \rangle - w_i \right\} , \quad s \in \mathbb{E}^*.
\end{align}
The conjugate function $c_*$ is $1 / L_f$ strongly convex and therefore $c$ is $L_f$ smooth.

On the other hand, from $\Psi \in \mathcal{P}(\mathcal{Z})$ and the $L_f$-smoothness of $f$ we have that $P_{L_f, v_i, w_i}(y) \geq l(y)$ for all $y \in \mathbb{E}$, $i = 1, ..., m + 1$ with $P_{L_f, v_i, w_i}(z_i) = f_i = l(z_i)$ and $\grad{P}_{L_f, v_i, w_i}(z_i) = g_i = \grad{l}(z_i)$ for all $i = 1, ..., m$. Therefore,
\begin{equation} \label{label_050}
c(z_i) = f_i, \quad \grad{c}(z_i) = g_i, \quad i = 1, ..., m.
\end{equation}
The $L_f$-smoothness of $c$ together with \eqref{label_050} implies that $c \in \mathcal{S}_{\mathbb{E}, L_f}(\mathcal{I}_{\mathcal{Z}})$. From the definition of $c$ and Proposition~\ref{label_021} it follows that
\begin{equation}
P_{L_f, v, w}(y) \geq c(y) \geq p(y), \quad y \in \mathbb{E}.
\end{equation}
Since we have assumed $\Psi = P_{L_f, v, w}$ to be arbitrary, we have that
\begin{equation} \label{label_051}
\hat{p}(y) = \min_{\Psi \in \mathcal{P}(\mathcal{Z})} \Psi(y) \geq p(y), \quad y \in \mathbb{E} .
\end{equation}
Combining \eqref{label_049} and \eqref{label_051} gives our result.
\end{proof}
An example of how our bound relates geometrically to the piece-wise lower linear model $l$ is shown in Figure~\ref{label_052}. In (a) we show that our bound can be interpreted as a form of smoothing. However, unlike the technique employed in \cite{ref_019}, which always produces a \emph{lower} bound on $l$, our bound is a smooth function that \emph{upper bounds} $l$ (see Proposition~\ref{label_021}). In (b) we show how our bound can be interpreted as a lower envelope of all upper parabolae on $l$.

\begin{figure}[h] \centering \footnotesize
\begin{minipage}[t]{0.48\linewidth} \centering
\includegraphics[width=\textwidth]{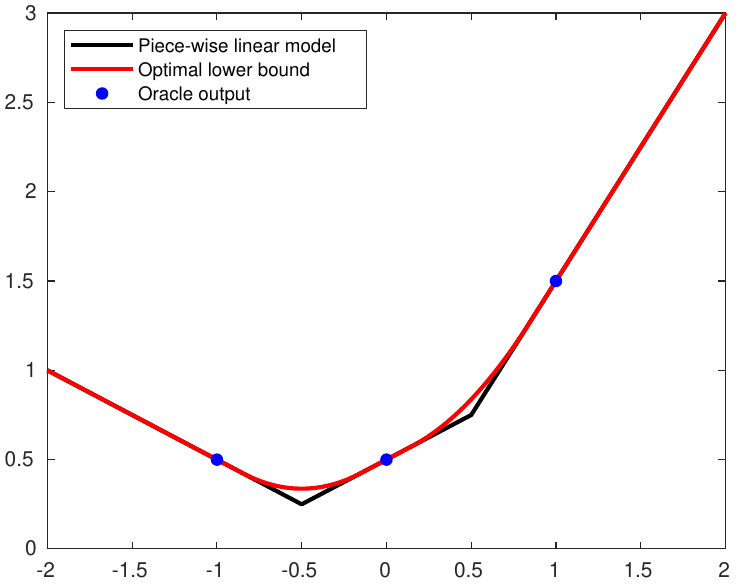}
(a) The optimal lower bound as a smoothing of the piece-wise linear model $l$
\end{minipage}
\hspace{2mm}
\begin{minipage}[t]{0.48\linewidth} \centering
\includegraphics[width=\textwidth]{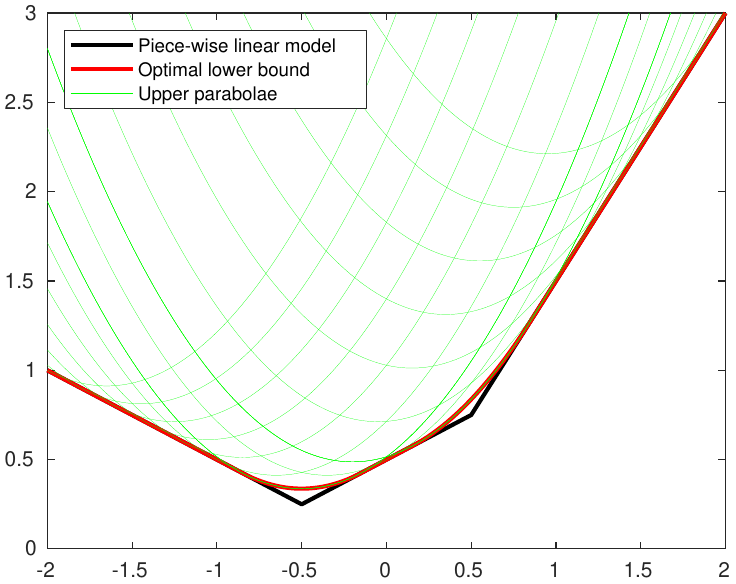}
(b) The optimal lower bound as the lower envelope of the parabolae dominating $l$.
\end{minipage}
\caption{An optimal interpolating lower bound corresponding to the oracle information \newline{} $\mathcal{I}_{\mathcal{Z}} = \{ (-1,0.5,-0.5), (0,0.5,0.5), (1,1.5,1.5)\}$}
\label{label_052}
\end{figure}

In the remainder of this work, we explore how our new bounds can be integrated in first-order schemes with the aim of improving their performance.

\section{An Improved Gradient Method with Memory} \label{label_053}

The first-order methods that are the most effective at utilizing the collected oracle information are the Gradient Methods with Memory, employing a piece-wise linear lower model~\cite{ref_021,ref_007,ref_008}. In Proposition~\ref{label_021} we have seen that the optimal bounds are tighter, if computed with reasonable accuracy. Moreover, in Subsection~\ref{label_030} we have shown that an arbitrarily small subset of the oracle history still produces valid bounds.
These properties suggest that our bound can, under certain conditions, \emph{directly} substitute the piece-wise linear model in memory methods.

One instance in which this assumption holds is the augmentation of the efficient Gradient Method with Memory, a fixed-point scheme described in \cite[Algorithm~1]{ref_008}. Note that when dealing with optimization methods, we use a notation similar to the one in \cite{ref_008}, with the model at a given point comprising $\tilde{m}$ records, stored as $H \in \mathbb{R}^{\tilde{m}}$, $\tilde{G} \in \mathbb{E} \times \mathbb{R}^{\tilde{m}}$.

To construct a fixed-point scheme, we impose that the first record be $\tilde{G}_1 = \grad{f}(\bar{x})$, and
$H_1 = f(\bar{x}) - \langle \grad{f}(\bar{x}), \bar{x} \rangle$, where $\bar{x}$ is the most recent test point. The remainder of $H$ and $\tilde{G}$ can be a subset of the oracle history or an aggregate. To account for all possibilities, aside from the condition on the first record, we \emph{only} require that
\begin{equation} \label{label_054}
f(y) \geq \max_{\lambda \in \Delta_{\tilde{m}}} \left\{ \left\langle H + \frac{\tau_f}{2} \diag(Q) + \tilde{G}^T y, \lambda \right\rangle - \frac{\tau_f}{2} \langle \lambda, Q \lambda \rangle \right\}, \quad y \in \mathbb{E},
\end{equation}
where $\tau_f \defq 1 / L_f$ , $Q \defq \tilde{G}^T B^{-1} \tilde{G}$, $\diag(Q)$ is the diagonal vector of $Q$, given by $(\diag(Q))_i = \| \tilde{G}_i \|_*^2$, $i = 1, ..., \tilde{m}$. Subsection~\ref{label_030} shows how our model can be obtained from the entire oracle history to satisfy \eqref{label_054}, using the notation correspondence $H = -\tilde{f}_*$.

Similarly to Algorithm~1 in \cite{ref_008}, new iterates can be generated using the majori-zation-minimization (MM) paradigm~\cite{ref_022,ref_014}. The update rule in this case is given by
$x_{+} = \displaystyle\argmin_{y \in \mathbb{E}} \max_{\lambda \in \Delta_{\tilde{m}}} u_{\lambda}(y)$, where
\begin{equation} \label{label_055}
u_{\lambda}(y) \defq \left\langle H + \frac{\tau_f}{2} \diag(Q) + \tilde{G}^T y, \lambda \right\rangle - \frac{\tau_f}{2} \langle \lambda, Q \lambda \rangle + \frac{1}{2 a}\| y - \bar{x} \|^2,
\end{equation}
with $a$ being the algorithm step size obtained though a line-search procedure outlined in the sequel. Strong duality holds and we have that $x_+ = \bar{x} - a B^{-1} \tilde{G} \lambda^*$, where
\begin{equation} \label{label_056}
\lambda^* = \argmin_{\lambda \in \Delta_{\tilde{m}}} \left\{ \frac{a + \tau_f}{2} \langle \lambda, Q \lambda \rangle - \left\langle H + \frac{\tau_f}{2} \diag(Q) + \tilde{G}^T \bar{x}, \lambda \right\rangle \right\}.
\end{equation}
The auxiliary problem in \eqref{label_056} is a Quadratic Program with a structure very similar to the one in \cite{ref_008}. It may not allow an exact solution.

We consider a subprocess $\lambda_{C,D}(A)$ that produces an \emph{approximate} solution $\lambda$ to the more general problem $\min_{\lambda \in \Delta_{\tilde{m}}} \left\{ \frac{A}{2} \langle \lambda, C \lambda \rangle - \langle D, \lambda \rangle \right\}$, with $d_{C, D}(A, \lambda) \defq \frac{A}{2} \langle \lambda, C \lambda \rangle - \langle D, \lambda \rangle$ denoting the objective value. An approximate solution to \eqref{label_056} is thus given by $\lambda_{Q,R}(a + \tau_f)$ where $R = H + \frac{\tau_f}{2} \diag(Q) + \tilde{G}^T \bar{x}$.

The line-search procedure for $a$, according to the MM principle, ensures that $u_{\lambda}(x_+)$ dominates $f(x_+)$. Since we allow inexact solutions $\lambda$ from \eqref{label_056} of arbitrary quality, line-search may not terminate. To remedy this, we simply revert to a Gradient Method step (setting $a = \tau_f$, $\lambda_1 = 1$ and $\lambda_i = 0$, $i = 2, ..., \tilde{m}$) if $a$ happens to fall below $t_f$. With the model changing from one iteration to the next, we index the quantities appropriately and discard the model reduction notation. For instance, the tuple $(H, \tilde{G})$ becomes $(H_{k + 1}, G_{k + 1})$. The resulting method is listed in Algorithm~\ref{label_057}.

\begin{algorithm}[h!]
\caption{An Improved Gradient Method with Memory}
\label{label_057}
\begin{algorithmic}[1]
\STATE \textbf{Input:} $B \succ 0$, $x_0 \in \mathbb{E}$, $L \geq L_f$, $r_u > 1 \geq r_d > 0$, $T \in \{1, 2, ..., +\infty \}$
\STATE $\tau = 1 / L$, $a_0 = \tau$
\FOR{$k = 0,\ldots{},T-1$}
\STATE $g_{k + 1} = \grad{f}(x_k)$
\STATE $h_{k + 1} = f(x_k) - \langle g_{k + 1}, x_k \rangle$ \label{label_058}
\STATE Generate $G_{k + 1}$ and $H_{k + 1}$ to include $g_{k + 1}$ and $h_{k + 1}$
\STATE Generate $Q_{k + 1}$ to equal $G_{k + 1}^T B^{-1} G_{k + 1}$
\STATE $R_{k + 1} = H_{k + 1} + \frac{\tau}{2} \diag(Q_{k + 1}) + G_{k + 1}^T x_k$
\STATE $a_{k + 1} = a_k / r_d$
\LOOP
\IF{$a_{k + 1} < \tau$}
\STATE $a_{k + 1} = \tau$
\STATE $x_{k + 1} = x_k - \tau B^{-1} g_{k + 1}$
\STATE Break from loop
\ENDIF
\STATE $\lambda_{k + 1} := \lambda_{Q_{k + 1}, R_{k + 1}}(a_{k + 1} + \tau)$
\STATE $x_{k + 1} := x_k - a_{k + 1} B^{-1} G_{k + 1} \lambda_{k + 1}$
\IF {$f(x_{k + 1}) \leq -d_{Q_{k + 1}, R_{k + 1}}(a_{k + 1} + \tau, \lambda_{k + 1})$}
\STATE Break from loop
\ENDIF
\STATE $a_{k + 1} = a_k / r_u$
\ENDLOOP
\ENDFOR
\end{algorithmic}
\end{algorithm}

To analyze the convergence of Algorithm~\ref{label_057} we need the following result. For notational simplicity, we temporarily remove the indices when we consider each iteration separately.
\begin{proposition} \label{label_059}
If the point $x_+$ and the step size $a$ be generated by one iteration of Algorithm~\ref{label_057}, then
\begin{equation} \label{label_060}
\frac{1}{2}\| x_+ - y \|^2 \leq \frac{1}{2}\| \bar{x} - y \|^2 + a \left( f(y) - f(x_+) \right), \quad y \in \mathbb{E} .
\end{equation}
\end{proposition}
\begin{proof}
We distinguish two cases. If line-search succeeds, then $f(x_+) \leq u_\lambda(x_+) = -d_{Q, R}(a + \tau_f, \lambda)$ where $\lambda$ is obtained from the subprocedure $\lambda_{Q,R}(a + \tau_f)$. The function $u_\lambda(y)$ is a parabola and hence can be written in canonical form as $u_\lambda(y) = u_\lambda(x_+) + \frac{1}{2 a} \| y - x_+ \|^2$ for all $y \in \mathbb{E}$, noting that $u^*_{\lambda} = u{_\lambda}(x_+)$.
On the other hand, \eqref{label_017} implies that $u_{\lambda}(y) \leq p(y) + \frac{1}{2 a} \| y - \bar{x} \|^2$.
From Proposition~\ref{label_021} and the line-search condition we have that
\begin{equation} \label{label_061}
f(y) + \frac{1}{2 a} \| y - \bar{x} \|^2 \geq u_{\lambda}(y) \geq f(x_+) + \frac{1}{2 a} \| y - x_+ \|^2 .
\end{equation}
Rearranging \eqref{label_061} gives \eqref{label_060}.

When line-search fails, we have $a = \tau_f$ and $f(x_+) \leq u_{\bar{x}}(\bar{x})$ where
\begin{equation}
x_+ = \argmin_{y \in \mathbb{E}} \left\{ u_{\bar{x}}(y) \defq f(\bar{x}) + \langle \grad{f}(\bar{x}), y - \bar{x}\rangle + \frac{1}{2 \tau_f} \| y - \bar{x} \|^2 \right\}.
\end{equation}
Using the same reasoning in the first case, we obtain \eqref{label_060}.
\end{proof}
Proposition~\ref{label_059} implies, in the same manner as shown in Theorem 2.3 in \cite{ref_008}, that the worst-case rate of Algorithm~\ref{label_057} is given by
\begin{equation}
f(x_k) - f^* \leq \frac{\| x_0 - x^*\|^2}{2 A_k} \leq \frac{L \| x_0 - x^*\|^2}{2 k}, \quad k \geq 1, \quad x^* \in X^*,
\end{equation}
where
where $A_k \defq \sum_{i = 1}^{k} a_i$ is the convergence guarantee and $L \geq L_f$ is the value of the Lipschitz constant fed to the algorithm.

\section{An Optimized Gradient Method with Memory} \label{label_062}

We now to turn our attention to the application of our bound in conjunction with acceleration. The update rules of accelerated schemes can be obtained using the machinery of estimate functions (see, e.g.,\cite{ref_018,ref_020,ref_008}). These functions comprise a global lower bound at the optimum, incorporating all the relevant oracle information obtained up to that point, and a strongly convex term that can be made arbitrarily small. The estimate function lower bounds are constructed by putting together, either as a convex combination~\cite{ref_018,ref_020} or a limit thereof~\cite{ref_008}, simple global lower bounds, each usually obtained using the oracle information of a single iteration. Our optimal bound does not appear to be applicable directly to this framework. Thus, to fully exploit the theoretical properties of our bound in conjunction with acceleration, we need to devise a new mathematical object.

\subsection{Primal-dual estimate functions} \label{label_063}

Whereas our bound lacks separability when viewed in the dual form given by \eqref{label_017}, it \emph{does} separate into the simple terms \eqref{label_006}, discussed in Subsection~\ref{label_003}, when seen in the primal form \eqref{label_013}. To accommodate the primal-dual bounds, that are additionally parameterized by the dual variable $g$, we introduce the primal-dual estimate functions having the following structure:
\begin{equation} \label{label_064}
\psi_k(y, g) = A_k W_k(y, g) + \frac{1}{2} \| y - x_0 \|^2, \quad k \geq 0,
\end{equation}
where $W_k$ is a function jointly convex in $y \in \mathbb{E}$ and $g \in \mathbb{E}^*$, being additionally strongly convex or constant in $g$. For each $k \geq 0$, there exists a $g_k^* \in \mathbb{E}^*$ such that
\begin{equation} \label{label_065}
W_k(x^*, g_k^*) \leq f(x^*).
\end{equation}
Here, $A_k$ is the convergence guarantee and $x^*$ is a specific member of $X^*$, which we consider fixed. The estimate function originally used to derive FGM in \cite{ref_018} actually enforced $A_0 > 0$ and some accelerated schemes introduced afterwards also considered this scenario (see, e.g., \cite{ref_010}). However, for simplicity, we assume in this work that $A_0 = 0$. For this reason, we do not define $W_0$.

We also denote the estimate function optima $x$ and $g$ as $v_k$ and $s_k$, respectively. Our assumptions on $W_k$ ensure their existence, with $v_k$ being unique, for $k \geq 0$. When $k = 0$ we have $v_0 = x_0$ and we set $s_0 = 0$, the zero vector in $\mathbb{E}^*$.

A sufficient condition that ensures algorithmic convergence is the estimate sequence property (ESP), now given by
\begin{equation} \label{label_066}
A_k f(x_k) \leq \psi_k^* \defq \min_{y \in \mathbb{E}, g \in \mathbb{E}^*} \psi_k(y, g) = \psi_k(v_k, s_k), \quad k \geq 1,
\end{equation}
where $x_k$ is the main iterate of the algorithm. The proof follows directly from the above definitions
\begin{equation} \label{label_067}
\begin{aligned}
A_k f(x_k) &\leq \psi_k^* \overset{\eqref{label_066}}{\leq} \psi_k(x^*, g_k^*) \overset{\eqref{label_064}}{=} A_k W_k(x^*, g_k^*) + \frac{1}{2} \| x^* - x_0 \|^2 \\
&\overset{\eqref{label_065}}{\leq} A_k f(x^*) + \frac{1}{2} \| x_0 - x^* \|^2, \quad x \in X^*, \quad k \geq 1.
\end{aligned}
\end{equation}

\subsection{Majorization-minimization} \label{label_068}

Smooth objectives are particularly appropriate for MM algorithms because the Lipschitz gradient property in \eqref{label_005} specifies the existence of a parabolic majorant at every point. The most straightforward means of applying the MM framework to a first-order scheme can be found in the Gradient Method, where every new iterate can be considered to be given by the optimum of the majorant in \eqref{label_005} at its predecessor. In the same manner as the Fast Gradient Method (FGM), we relax this condition to allow the main iterate $x_k$ at every iteration $k$ to be the optimum of the majorant $u_k$ at an auxiliary point $y_k$. The update rule for all $k \geq 1$ becomes
\begin{align}
x_k &= \arg\min_{y \in \mathbb{E}} \left\{u_k \defq f(y_k) + \langle \grad{f}(y_k), y - y_k \rangle + \frac{L_f}{2} \| y - y_k \|^2 \right\} \label{label_069} \\
& = y_k - \frac{1}{L_f} B^{-1} \grad{f}(y_k). \label{label_070}
\end{align}
As in the case of FGM, we consider the first step to be a gradient step and thus have $y_1 = x_0$. For convenience, we also set $y_0 = x_0$, although this will never actually be used by any algorithm.

From \eqref{label_005} and \eqref{label_070} we obtain the well-known descent rule
\begin{equation} \label{label_071}
f(x_k) \leq f(y_k) - \frac{1}{2 L_f} \| \grad{f}(y_k) \|_*^2, \quad k \geq 1.
\end{equation}
This rule provides an stricter alternative to the estimate sequence property in \eqref{label_066} in the form of
\begin{equation} \label{label_072}
A_k \left( f(y_k) - \frac{1}{2 L_f} \| \grad{f}(y_k) \|_*^2 \right) \leq \psi_k^*, \quad k \geq 1.
\end{equation}
Since \eqref{label_072} implies \eqref{label_066}, \eqref{label_072} guarantees convergence according to \eqref{label_067}. It is however less computationally complex, as $\grad{f}(y_k)$ may share subexpressions and can be computed concurrently with $\grad{f}(y_k)$.

We introduce some additional notation. Let $\Gamma_k$ be the estimate function gaps in \eqref{label_072}, given by $\Gamma_k \defq \psi_k^* - A_k f(y_k) + \frac{A_k}{2 L_f} \| \grad{f}(y_k) \|_*^2$, $k \geq 1$. Therefore, \eqref{label_072} ensures the nonnegativity of the gaps.

\subsection{A simple optimal scheme} \label{label_073}

Next, we seek to determine an update rule for $y_k$ that renders our method optimal with respect to the available collected information.

To this end we employ the algorithmic design pattern presented in \cite{ref_010,ref_005}. Instead of obtaining the global lower bounds in the estimate functions by weighted averaging simple \emph{primal} lower bounds, we choose \emph{primal-dual} bounds. Namely, the estimate function update now needs to satisfy
\begin{equation} \label{label_074}
\psi_{k + 1}(y, g) = \psi_k(y, g) + a_{k + 1} w_{k + 1}(y, g), \quad y \in \mathbb{E}, \quad g \in \mathbb{E}^*, \quad k \geq 0,
\end{equation}
where
\begin{equation} \label{label_075}
w_{k + 1}(y, g) = f(y_{k + 1}) + \langle \grad{f}(y_{k + 1}), y - y_{k + 1} \rangle + \frac{1}{2 L_f} \| g - \grad{f}(y_{k + 1}) \|_*^2.
\end{equation}
In the context of this simple method we have $A_{k + 1} = A_k + a_{k + 1}$ with $a_{k + 1} > 0$. We have now made all the necessary assumptions needed to construct an optimal method.

The simple lower bounds in \eqref{label_075} are linear in $x$ and parabolic in $g$ and we therefore have the following canonical form for all estimate functions
\begin{equation} \label{label_076}
\psi_k(y, g) = \psi_k^* + \frac{1}{2} \| y - v_k \|^2 + \frac{A_k}{2 L_f} \| g - s_k \|_*^2, \quad k \geq 0.
\end{equation}
We substitute the simple lower bound in \eqref{label_075} and the canonical form of the estimate function in \eqref{label_076} within \eqref{label_074} and obtain the following expression that holds for all $k \geq 0$, from which \emph{all} update rules can be derived:
\begin{equation} \label{label_077}
\begin{gathered}
\psi_{k + 1}^* + \frac{1}{2} \| y - v_{k + 1}\|^2 + \frac{A_{k + 1}}{2 L_f} \| g - s_{k + 1} \|_*^2 = \psi_k^* + \frac{1}{2} \| y - v_{k}\|^2 \\ + \frac{A_k}{2 L_f} \| g - s_{k} \|_*^2 + a_{k + 1} \left( f(y_{k + 1}) + \langle \grad{f}(y_{k + 1}), y - y_{k + 1} \rangle + \frac{1}{2 L_f} \| g - \grad{f}(y_{k + 1}) \|_*^2 \right).
\end{gathered}
\end{equation}

First, by differentiating \eqref{label_077} with respect to $y$ we obtain
\begin{equation} \label{label_078}
v_{k + 1} = v_k - a_{k + 1} B^{-1} \grad{f}(y_{k + 1}), \quad k \geq 0.
\end{equation}
The remainder of the update rules can be easily obtained from the following result.
\begin{theorem} \label{label_079}
For any algorithm that employs primal dual estimate functions updated as in \eqref{label_074} using the simple lower bounds given by \eqref{label_075}, we have
\begin{equation}
\begin{gathered}
\Gamma_{k + 1} \geq \Gamma_k + \left( \frac{A_k + A_{k + 1}}{2 L_f} - \frac{a_{k + 1}^2}{2} \right) \| \grad{f}(y_{k + 1}) \|_*^2 \\ + \left\langle \grad{f}(y_{k + 1}), A_k \left(y_k - \frac{1}{L_f}B^{-1}\grad{f}(y_k)\right) + a_{k + 1} v_k - A_{k + 1} y_{k + 1} \right\rangle, \quad k \geq 0,
\end{gathered}
\end{equation}
where $A_{k + 1} \defq A_k + a_{k + 1}$.
\end{theorem}
\begin{proof}
Herein we consider all $k \geq 0$, unless specified otherwise.
Taking $x = y_{k + 1}$, $g = \grad{f}(y_{k + 1})$ in \eqref{label_077}, noting that $w_k(y_k, \grad{f}(y_k)) = f(y_k)$, we obtain
\begin{equation} \label{label_080}
\begin{gathered}
\psi_{k + 1}^* + \frac{1}{2} \| y_{k + 1} - v_{k + 1}\|^2 + \frac{A_{k + 1}}{2 L_f} \| \grad{f}(y_{k + 1}) - s_{k + 1} \|_*^2 \\= \psi_k^* + \frac{1}{2} \| y_{k + 1} - v_{k}\|^2 + \frac{A_k}{2 L_f} \| \grad{f}(y_{k + 1}) - s_{k} \|_*^2 + a_{k + 1} f(y_{k + 1}).
\end{gathered}
\end{equation}
Expanding and rearranging terms in \eqref{label_080} produces
\begin{equation} \label{label_081}
\begin{gathered}
\psi_{k + 1}^* - A_{k + 1} f(y_{k + 1}) + \frac{A_{k + 1}}{2 L_f} \| \grad{f}(y_{k + 1}) \|_*^2 = \psi_k^* - A_k f(y_{k})\\
+ A_k ( f(y_{k}) - f(y_{k + 1}) ) + \frac{A_{k + 1}}{2 L_f} \| \grad{f}(y_{k + 1}) \|_*^2 + \mathcal{S}_{k + 1}
+ \mathcal{V}_{k + 1},
\end{gathered}
\end{equation}
where $\mathcal{S}_{k + 1}$ and $\mathcal{V}_{k + 1}$ are, respectively defined as
\begin{align}
\mathcal{S}_{k + 1} &\defq \frac{A_{k}}{2 L_f} \| \grad{f}(y_{k + 1}) - s_{k} \|_*^2 - \frac{A_{k + 1}}{2 L_f} \| \grad{f}(y_{k + 1}) - s_{k + 1} \|_*^2, \label{label_082} \\
\mathcal{V}_{k + 1} &\defq \frac{1}{2} \| y_{k + 1} - v_{k}\|^2 - \frac{1}{2} \| y_{k + 1} - v_{k + 1}\|^2 \label{label_083} .
\end{align}
Differentiating \eqref{label_077} with respect to $g$ also taking $x = y_{k + 1}$, $g = \grad{f}(y_{k + 1})$ yields
\begin{equation} \label{label_084}
A_{k + 1} (\grad{f}(y_{k + 1}) - s_{k + 1}) = A_k (\grad{f}(y_{k + 1}) - s_{k}).
\end{equation}
From \eqref{label_084} it follows that $\mathcal{S}_{k + 1}$ is always nonnegative, namely
\begin{equation} \label{label_085}
\mathcal{S}_{k + 1} = \left(\frac{A_k}{2 L_f} - \frac{A_k^2}{2 L_f A_{k + 1}} \right) \| \grad{f}(y_{k + 1}) - s_k\|_*^2 = \frac{a_{k + 1} A_k}{2 L_f A_{k + 1}} \| \grad{f}(y_{k + 1}) - s_k\|_*^2 \geq 0 .
\end{equation}
Using \eqref{label_078} we can refactor $\mathcal{V}_{k + 1}$ as
\begin{equation} \label{label_086}
\begin{aligned}
\mathcal{V}_{k + 1} &= \frac{1}{2} \langle (y_{k + 1} - v_{k}) - (y_{k + 1} - v_{k + 1}), B \left( (y_{k + 1} - v_{k}) + (y_{k + 1} - v_{k + 1}) \right) \rangle \\
&= - \frac{1}{2} \langle a_{k + 1} B^{-1} \grad{f}(y_{k + 1}), B( 2 (y_{k + 1} - v_{k}) ) + a_{k + 1} B^{-1} \grad{f}(y_{k + 1}) \rangle \\
&= \langle \grad{f}(y_{k + 1}), a_{k + 1} v_{k} - a_{k + 1} y_{k + 1} \rangle - \frac{a^2_{k + 1}}{2} \| \grad{f}(y_{k + 1}) \|_*^2.
\end{aligned}
\end{equation}

On the other hand, multiplying $f(y_k) \geq w_{k + 1}\left(y_k, \grad{f}(y_{k})\right)$ by $A_k$ (when $k = 0$ this amounts to $0 \geq 0$ without having $y_0$ defined) and expanding terms gives
\begin{equation} \label{label_087}
\begin{gathered}
A_k( f(y_{k}) - f(y_{k + 1}) ) \geq \langle \grad{f}(y_{k + 1}), A_k y_k - A_k y_{k + 1} \rangle + \frac{A_k}{2 L_f} \| \grad{f}(y_{k}) \|_*^2 \\ + \frac{A_k}{2 L_f} \| \grad{f}(y_{k + 1}) \|_*^2 - \left\langle \grad{f}(y_{k + 1}), \frac{A_k}{L_f} B^{-1} \grad{f}(y_{k}) \right\rangle.
\end{gathered}
\end{equation}

Applying \eqref{label_085}, \eqref{label_086} and \eqref{label_087} in \eqref{label_081} and regrouping terms gives the desired result.
\end{proof}

There is an outstanding resemblance between Theorem~\ref{label_079} and \cite[Theorem 6]{ref_005}, especially concerning the update rules. Following the design procedure of the Accelerated Composite Gradient Method (ACGM) in \cite{ref_005}, we can ensure that Theorem~\ref{label_079} holds for any algorithm satisfying \eqref{label_070} as well as the following update rules for all $k \geq 0$:
\begin{gather}
A_{k + 1} y_{k + 1} = A_k x_k + a_{k + 1} v_k, \label{label_088} \\
L_f a_{k + 1}^2 \leq 2 A_k + a_{k + 1}. \label{label_089}
\end{gather}
Having fixed the main iterate update in \eqref{label_070}, we see it reappear in \eqref{label_088}. Under the assumption in \eqref{label_070}, the update rule for the auxiliary point is identical to the one derived in \cite[Theorem 6]{ref_005}, which in turn matches the one in the original formulation of the Fast Gradient Method \cite{ref_017}.

The estimate function lower bound is given by $W_k(y, g) = \frac{1}{A_k} \sum_{i = 1}^{k} a_i w_i(y, g)$, $y \in \mathbb{E}$, $g \in \mathbb{E}^*$, $k \geq 1$ satisfying $W_k(x^*, 0) \leq f(x^*)$ for all $x^* \in X^*$. Combining \eqref{label_067} with \eqref{label_071} we have that
\begin{equation}
f(x_k) - f(x^*) \leq \frac{1}{A_k} \psi_k^* - f(x^*) \leq \frac{1}{2 A_k} \|x_0 - x^*\|^2, \quad k \geq 1.
\end{equation}
The largest convergence guarantees are obtained by enforcing equality in \eqref{label_089}, thereby obtaining
\begin{equation} \label{label_090}
a_k = \frac{k}{L_f}, \quad A_k = \frac{k (k + 1)}{2 L_f}, \quad k \geq 1,
\end{equation}
yielding a worst-case rate of
\begin{equation} \label{label_091}
f(x_k) - f(x^*) \leq \frac{ L_f \|x_0 - x^*\|^2}{k(k + 1)}, \quad x \in X^*, \quad k \geq 1.
\end{equation}
This rate, even up to the proportionality constant, may be the best attainable on this problem class by a black-box method~\cite{ref_013,ref_003}. In fact, our simple optimal method closely resembles the Optimized Gradient Method~\cite{ref_004,ref_012,ref_002}, previously derived using the Performance Estimation Framework.

However, unlike in FGM and extensions such as ACGM, the rigid nature of the estimate function gaps in Theorem~\ref{label_079} and the weight update in \eqref{label_089} hinder the direct application of fully-adaptive line-search, wherein the Lipschitz estimates can both increase and decrease. Instead, we can increase the convergence guarantee $A_k$ directly, by closing the gap in \eqref{label_072}. Moreover, by relaxing the assumption in \eqref{label_074} we can increase $A_k$ further.

\subsection{Adding memory to the Optimized Gradient Method}

We construct a memory model similar to the one discussed in Section~\ref{label_053}. We store $\tilde{m}$ records in the pair $\bar{H} \in \mathbb{R}^{\tilde{m}}$ and $\tilde{G} \in \mathbb{E} \times \mathbb{R}^{\tilde{m}}$. Based on the findings in Subsection~\ref{label_040}, we propose normalized estimate functions with memory taking for all $y \in \mathbb{E}$ and $g \in \mathbb{E}^*$ the form $\tilde{\omega}(y, g) = \max_{\lambda \in \Delta_{\tilde{m}}} \omega_{\lambda}(y, g)$ where
\begin{equation} \label{label_092}
\omega_{\lambda}(y, g) \defq \langle \bar{H} + \tilde{G}^T (y - \tau_f B^{-1} g), \lambda \rangle + \frac{\tau_f}{2} \| g \|_*^2 + \frac{1}{2 A} \| y - x_0 \|^2, \quad \lambda \in \Delta_{\tilde{m}}.
\end{equation}
Note that the normalized functions require $A > 0$ and cannot be used in the first iteration. We can ensure that $A \tilde{\omega}$ is a valid primal-dual estimate function, according to the criteria outlined in Subsection~\ref{label_063} if there exists a linear transform $\mathcal{T} \in \mathbb{R}^{m} \times \mathbb{R}^{\tilde{m}}$, with $\mathcal{T}_i \in \Delta_{\tilde{m}}$ for every $i = 1, ..., \tilde{m}$, such that $\bar{H} = -\mathcal{T}^T \bar{f}_*$ and $\tilde{G} = G \mathcal{T}$.

The verification of the stricter estimate sequence property in \eqref{label_072} requires the calculation of $\omega^*$, which, by means of strong duality becomes
\begin{align}
\omega^* &= \min_{y \in \mathbb{E}, g \in \mathbb{E}^*} \max_{\lambda \in \Delta_{\tilde{m}}} \omega_{\lambda}(y, g) = \max_{\lambda \in \Delta_{\tilde{m}}} \min_{y \in \mathbb{E}, g \in \mathbb{E}^*} \omega_{\lambda}(y, g) \\
&= \max_{\lambda \in \Delta_{\tilde{m}}} \left\{\langle \bar{H} + \tilde{G}^T x_0, \lambda \rangle - \frac{A + \tau_f}{2} \langle \lambda, Q \lambda \rangle \right\}, \label{label_093}
\end{align}
where $Q = \tilde{G}^T B^{-1} \tilde{G}$. By taking the minus sign in the objective of \eqref{label_093}, this auxiliary problem becomes the Quadratic Program $\min_{\lambda \in \Delta_{\tilde{m}}} d_{Q, S}(\lambda, A + \tau_f)$ where $S \defq \bar{H} + \tilde{G}^T x_0$.

Despite its simplicity, the auxiliary problem cannot be solved exactly and we consider the approximate solution $\lambda$ returned by a subprocedure $\lambda_{Q, S}(A + \tau_f)$ and the resulting function $\omega_{\lambda}$, which we shorten to $\omega$, with $\omega^* = -d_{Q, S}(\lambda, A + \tau_f)$. The only requirement we place on the subprocedure $\lambda_{Q, S}(A + \tau_f)$ is that it should output a vector $\lambda$ giving an objective value no worse that the value at the starting point. Most estimate sequence based auxiliary optimization schemes, including higher order ones, satisfy this condition by default.

The increase of the convergence guarantee can be accomplished iteratively using Newton's rootfinding algorithm as previously shown in \cite{ref_008}. Let the ESP gap in \eqref{label_072} be given by $A \gamma(A)$ with $\gamma(A) \defq \omega^* - \bar{e}$, where $\bar{e} \defq f(\bar{y}) - \frac{\tau_f}{2} \| \grad{f}(\bar{y}) \|_*^2$. Increasing the convergence guarantee to the highest level possible amounts to seeking the zero intercept of $\gamma$. Using Danskin's lemma we obtain the simple Newton update rule
$A^{(t + 1)} = A^{(t)} + 2 \frac{\gamma(A^{(t)})}{\langle \lambda(A^{(t)}), Q \lambda(A^{(t)}) \rangle}$,
which is identical in form with the one used by AGMM in \cite{ref_008}. The inexactness in obtaining $\lambda(A)$ invalidates the convergence guarantees of the rootfinding scheme but as long as the gaps are positive the updates do increase $A$. The increase can simply be halted when the gap is no longer positive.

At the starting state, the main algorithm does not have a history of oracle calls available and instead performs a GM step. We thus have
\begin{equation}
\omega_1(y, g) = f(x_0) + \langle \grad{f}(x_0), y - x_0 - \tau_f B^{-1} g \rangle + \frac{\tau_f}{2} \| g \|_*^2 + \frac{1}{2 A_1} \| y - x_0 \|^2,
\end{equation} yielding $\omega^* = f(x_0) - \frac{A_1}{2} \| \grad{f}(x_0) \|_*^2$. The ESP in \eqref{label_072} is satisfied with equality for $A_1 = a_1 = \tau_f$.

When $k \geq 1$, at the beginning of each iteration $k$, we have the previous normalized estimate function $\omega_k$ already computed, along with $\lambda_k$ and $A_k$. Function $\omega_k$ is a parabola in $y$ determined uniquely by 3 parameters: the scalars $A_k$ and $h_k \defq \langle \lambda_k, \bar{H}_k \rangle$ as well as the vector $g_k \defq G_k \lambda_k$. We write $\omega_k$ as
\begin{equation} \label{label_094}
\omega_k(y, g) = h_k + \langle g_k, y - \tau_f B^{-1} g \rangle + \frac{\tau_f}{2} \| g \|_*^2 + \frac{1}{2 A_k} \| y - x_0 \|^2, \quad y \in \mathbb{E}, \quad g \in \mathbb{E}^*.
\end{equation}

To create the next iterate, we obtain the new weight $a_{k + 1}$ as in the memory-less case using $L_f a_{k + 1}^2 = 2 A_k + a_{k + 1}$ and compute the auxiliary point $y_{k + 1}$ using \eqref{label_088} where $v_k$ is now obtained from $(v_k, s_k) \defq \argmin_{y \in \mathbb{E}, g \in \mathbb{E}^*} \omega_k(y, g)$ and is given by $v_k = x_0 - A_{k} B^{-1} g_k \lambda_k$. We next create a majorant $u_{k + 1}$ using \eqref{label_069}, the minimizer of which is the next iterate, as given by \eqref{label_070}.

Now we can construct the next normalized estimate function. As in AGMM~\cite{ref_008}, we impose constraints on the first two entries in the new model:
\begin{equation} \label{label_095}
\begin{aligned}
H_{k + 1}^{(1)} &= h_k,& H_{k + 1}^{(2)} &= \bar{h}_{k + 1} \defq f(y_{k + 1}) - \langle \grad{f}(y_{k + 1}), y_{k + 1} \rangle + \frac{\tau_f}{2} \| \grad{f}(y_{k + 1}) \|_*^2, \\
G_{k + 1}^{(1)} &= g_k,& G_{k + 1}^{(2)} &= \bar{g}_{k + 1} \defq \grad{f}(y_{k + 1}).
\end{aligned}
\end{equation}
The worst $\lambda_{k + 1}$ that we can afford is likewise given by
\begin{equation} \label{label_096}
(A_k + a_{k + 1}) \lambda_{k + 1}^{(0)} = (A_k, a_{k + 1}, 0, ..., 0)^T.
\end{equation}
We choose $\lambda_{k + 1}^{(0)}$ as the starting point of subprocedure $\lambda_{Q_{k + 1}, S_{k + 1}}(A_{k + 1} + \tau_f)$. The resulting method is listed in Algorithm~\ref{label_097}, which in turn calls the convergence guarantee adjustment procedure listed in Algorithm~\ref{label_099}.

\begin{algorithm}[h]
\caption{An Optimized Gradient Method with Memory}
\label{label_097}
\begin{algorithmic}[1]
\STATE \textbf{Input:} $B \succ 0$, $x_0 \in \mathbb{E}$, $L \geq L_f$, $T \in \{1, 2, ..., +\infty\}$
\STATE $v_0 = x_0$, $A_0 = 0$, $\tau = 1 / L$
\FOR{$k = 0,\ldots{},T-1$}
\STATE $a_{k + 1} = \frac{\tau}{2}\left(1 + \sqrt{1 + 4 L A_k} \right)$ \label{label_098}
\STATE $y_{k + 1} = \frac{1}{A_k + a_{k + 1}}(A_k x_k + a_{k + 1} v_k)$\\[1mm]
\STATE $f_{k + 1} = f(y_{k + 1})$ \\[1mm]
\STATE $\bar{g}_{k + 1} = \grad{f}(y_{k + 1})$\\[1mm]
\STATE $x_{k + 1} = y_{k + 1} - \tau B^{-1} \bar{g}_{k + 1}$\\[2mm]
\STATE $\bar{h}_{k + 1} = f_{k + 1} - \langle \bar{g}_{k + 1}, y_{k + 1} \rangle + \frac{\tau}{2} \| \bar{g}_{k + 1} \|_*^2$\\[2mm]
\STATE $e_{k + 1} = f_{k + 1} - \frac{\tau}{2} \| \bar{g}_{k + 1} \|_*^2$\\[1mm]
\IF {$k = 0$}
\STATE $\bar{H}_1 = \bar{h}_{1}$, $G_1 = \bar{g}_1$, $Q_1 = \| \bar{g}_1 \|_*^2$, $S_1 = f_{1} + \frac{\tau}{2} \| \bar{g}_1 \|_*^2$, $\lambda_{1} = 1$, $A_1 = a_1$
\ELSE
\STATE Generate $\bar{H}_{k + 1}$ and $G_{k + 1}$ to satisfy \eqref{label_095}
\STATE Generate $Q_{k + 1}$ to equal $G_{k + 1}^T B^{-1} G_{k + 1}$
\STATE Generate $S_{k + 1}$ to equal $\bar{H}_{k + 1} + G_{k + 1}^T x_0$
\STATE $\lambda^{(0)}_{k + 1} = \frac{1}{A_k + a_{k + 1}}(A_k, a_{k + 1}, 0, ..., 0)^T$
\STATE $\lambda_{k + 1}, A_{k + 1} = \operatorname{Newton}(\tau, Q_{k + 1}, S_{k + 1}, \bar{e}_{k + 1}, \lambda^{(0)}_{k + 1}, A_k + a_{k + 1})$
\ENDIF
\STATE $h_{k + 1} = \langle \bar{H}_{k + 1}, \lambda_{k + 1} \rangle$
\STATE $g_{k + 1} = G_{k + 1} \lambda_{k + 1}$
\STATE $v_{k + 1} = x_0 - A_{k + 1} B^{-1} g_{k + 1}$
\ENDFOR
\end{algorithmic}
\end{algorithm}

\begin{algorithm}[h]
\caption{Newton$(\tau, Q, S, \bar{e}, \lambda^{(0)}, A^{(0)})$}
\label{label_099}
\begin{algorithmic}[1]
\STATE $\lambda_{\operatorname{valid}} := \lambda^{(0)}$
\STATE $A_{\operatorname{valid}} := A := A^{(0)}$
\FOR{$t = 0, \ldots{}, N - 1$}
\STATE $\lambda := \lambda_{Q, S}(A + \tau)$ with starting point $\lambda^{(0)}$
\STATE $\omega^* := \langle S, \lambda \rangle - \frac{A}{2} \langle \lambda, Q \lambda \rangle$
\IF {$\omega^* < e$}
\STATE Break from loop
\ENDIF
\STATE $\lambda_{\operatorname{valid}} := \lambda$
\STATE $A_{\operatorname{valid}} := A$
\STATE $A := A + 2 \frac{\omega^* - e}{\langle \lambda, Q \lambda \rangle}$
\ENDFOR
\RETURN $\lambda_{\operatorname{valid}}, A_{\operatorname{valid}}$
\end{algorithmic}
\end{algorithm}

The main result governing the convergence of Algorithm~\ref{label_097} is as follows.
\begin{proposition} \label{label_100}
Let the worst-case estimate function $\bar{\omega}_{k + 1}$, created during iteration $k \geq 1$, be given for all $y \in \mathbb{E}$ and $g \in \mathbb{E}^*$ by
\begin{equation} \label{label_101}
\bar{\omega}_{k + 1}(y, g) = \left\langle \bar{H}_{k + 1} + G_{k + 1}^T \left(y - \tau_f B^{-1} g \right), \lambda_{k + 1}^{(0)} \right\rangle + \frac{\tau_f}{2} \| g \|_*^2 + \frac{1}{2 \bar{A}_{k + 1}} \| y - x_0 \|^2,
\end{equation}
where $\bar{A}_{k + 1} = A_k + a_{k + 1}$. If the previous ESP holds as given by \eqref{label_072}, then next ESP also holds, even in the worst-case, namely
\begin{equation}
f(y_{k + 1}) - \frac{\tau_f}{2} \| \grad{f}(y_{k + 1}) \|_*^2 \leq \bar{\omega}_{k + 1}^*.
\end{equation}
\end{proposition}
\clearpage
\begin{proof}
With the constraints imposed on the model in \eqref{label_095} we expand \eqref{label_101} using \eqref{label_096} and we multiply both sides by $\bar{A}_{k + 1}$ to obtain
\begin{equation} \label{label_102}
\begin{gathered}
(A_k + a_{k + 1}) \bar{\omega}_{k + 1}(y, g) = A_k (h_k + \langle g_k, y - \tau_f B^{-1}g \rangle) \\
+ a_{k + 1}(\bar{h}_{k + 1} + \langle \grad{f}(y_{k + 1}), y - \tau_f B^{-1}g \rangle) + \frac{(A_k + a_{k + 1}) \tau_f}{2} \|g \|_*^2 + \frac{1}{2} \| y - x_0 \|_*^2 \\
\overset{\eqref{label_094}}{=} A_k \omega_k(y, g) + a_{k + 1} \left(f(y_{k + 1}) - \langle \grad{f}(y_{k + 1}), y_{k + 1} \rangle + \frac{\tau_f}{2} \| \grad{f}(y_{k + 1}) \|_*^2 + \frac{\tau_f}{2} \| g \|_*^2 \right) \\ + a_{k + 1} \langle \grad{f}(y_{k + 1}), y - \tau_f B^{-1} g \rangle
\overset{\eqref{label_075}}{=} A_k \omega_k(y, g) + a_{k + 1} w_{k + 1}(y, g) , \quad y \in \mathbb{E}, \quad g \in \mathbb{E}^*.
\end{gathered}
\end{equation}
The conditions in Theorem~\ref{label_079} are thus met and applying it completes the proof.
\end{proof}

Subprocedure $\lambda_{Q, S}(A + \tau_f)$ further ensures that $\omega_{k + 1}^* \leq \bar{\omega}_{k + 1}^*$ for all $k \geq 1$. Since the ESP in \ref{label_072} holds during the first iteration, Proposition~\ref{label_100} implies that the ESP holds for all iterations $k \geq 1$. It follows that Algorithm~\ref{label_097} also has a worst-case rate given by \eqref{label_091}, which is the best known on this problem class. Note that we allow the overestimation of $L_f$, in which case $L_f$ in \eqref{label_091} is replaced by $L \geq L_f$.

\section{Augmenting the estimate sequence}

The convergence of the original Optimized Gradient Method has been analyzed using potential functions in \cite{ref_002}. The functions, as well as the update rules themselves were obtained by manually fitting the numerical data obtained using the Performance Estimate Framework, which numerically simulates a resisting oracle. The convergence analysis itself bears a striking resemblance with the gap sequence proof structure described in \cite{ref_010,ref_005}. In this section we seek to establish the relation between the gap sequence obtained by augmenting the primal-dual estimate functions and the potential functions from \cite{ref_002}.

Recall that estimate functions contain an aggregate lower bound that should be a good approximation of the objective function \emph{at an optimal point}. Augmentation consists of forcibly \emph{closing the gap} between this approximation and the optimal value (see \cite{ref_005} for a detailed exposition).
When dealing with primal-dual estimate functions, it is not necessary to close this gap \emph{fully}. To eliminate the dual terms, which introduce the unnecessary slack $\mathcal{S}_{k + 1}$ in our analysis, we opt for \emph{primal} augmentation, yielding the augmented estimate functions as follows:
\begin{equation} \label{label_103}
\bar{\psi}_k(y, g) \defq \psi_k(y, g) + A_k \left( f(x^*) - W_k(x^*, s_k) \right), \quad y \in \mathbb{E}, \quad g \in \mathbb{E}^*, \quad k \geq 1.
\end{equation}
The augmented estimate sequence property (AESP) is thus given by
\begin{equation} \label{label_104}
A_k \left( f(y_k) - \frac{1}{2 L_f} \| \grad{f}(y_k) \|^2 \right) \leq \bar{\psi_k}^*, \quad k \geq 1.
\end{equation}
Note that while the augmentation described in \cite{ref_010} constitutes a relaxation, the ESP in \eqref{label_072} is not necessarily a stronger condition than the AESP in \eqref{label_104}.

If the new iterate is obtained through majorization-minimization using \eqref{label_070}, the AESP in \eqref{label_104} guarantees convergence as follows:
\begin{equation}
\begin{gathered}
A_k f(x_{k}) \overset{\eqref{label_070}}{\leq} A_k f(y_k) - \frac{A_k }{2 L_f} \| \grad{f}(y_k) \|^2 \overset{\eqref{label_104}}{\leq} \bar{\psi}_k^* \overset{\eqref{label_103}}{=} \psi_k^* + A_k \left(f(x^*) - W_k(x^*, s_k)\right) \\ \leq \psi_k(x^*, s_k) + A_k \left( f(x^*) - W_k(x^*, s_k) \right) \overset{\eqref{label_064}}{=} A_k f(x^*) + \frac{1}{2} \| x^* - x_0 \|^2 , \quad k \geq 1.
\end{gathered}
\end{equation}
\begin{proposition}
The augmented estimate sequence gaps
\begin{equation} \label{label_105}
\bar{\Gamma}_k \defq \bar{\psi}_k^* - A_k f(y_k) + \frac{A_k}{2 L_f} \| \grad{f}(y_k) \|^2, \quad k \geq 1.
\end{equation}
satisfy
\begin{equation}
\bar{\Gamma}_k = \mathcal{D}_0 - \mathcal{D}_k, \quad k \geq 1,
\end{equation}
where the gap terms $\mathcal{D}_k$ are defined as
\begin{equation}
\mathcal{D}_k = A_k \left( f(y_k) - \frac{1}{2 L_f} \| \grad{f}(y_k) \|^2 - f(x^*) \right) + \frac{1}{2} \| v_k - x^* \|^2, \quad k \geq 0.
\end{equation}
\end{proposition}
\begin{proof}
We can obtain a simple gap sequence from the augmented estimate sequence by using the definition and canonical form of the estimate function
\begin{equation} \label{label_106}
\psi_k^* + \frac{1}{2} \| v_k - x^* \|^2 \overset{\eqref{label_076}}{=} \psi_k(x^*, s_k) \overset{\eqref{label_064}}{=} A_k W_k(x^*, s_k) + \frac{1}{2} \| x_0 - x^* \|^2, \quad k \geq 1.
\end{equation}
From \eqref{label_103} and \eqref{label_106} we have for all $k \geq 1$
\begin{equation} \label{label_107}
\bar{\psi}_k^* = \psi_k^* + A_k \left(f(x^*) - W_k(x^*, s_k)\right) = A_k f(x^*) + \frac{1}{2} \left( \| v_k - x^* \|^2 - \| x_0 - x^* \|^2 \right).
\end{equation}
Expanding \eqref{label_105} using \eqref{label_107} completes the proof.
\end{proof}
Thus, as sufficient condition for the maintenance at runtime of the AESP is to have non-increasing gap terms.
\begin{theorem} \label{label_108}
For any method that employing the update
\begin{equation} \label{label_109}
v_{k + 1} = v_k - a_{k + 1} B^{-1} \grad{f}(y_{k + 1}), \quad k \geq 0,
\end{equation} but not necessarily the ESP in \eqref{label_072}, we have that
\begin{equation}
\begin{gathered}
\mathcal{D}_{k} \geq \mathcal{D}_{k + 1} + \left( \frac{A_{k + 1}}{L_f} - \frac{a_{k + 1}^2}{2} \right) \| \grad{f}(y_{k + 1}) \|^2 \\ + \left\langle \grad{f}(y_{k + 1}), A_k \left( y_k - \frac{1}{L_f} B^{-1} \grad{f}(y_{k}) \right) + a_{k + 1} v_k - A_{k + 1} y_{k + 1} \right\rangle, \quad k \geq 0.
\end{gathered}
\end{equation}
\end{theorem}
\begin{proof}
We closely follow the reasoning in the proof of \cite[Theorem 3]{ref_005} and define the residual $\mathcal{R}_{k + 1}(y) \defq f(y) - w_{k + 1}(y, \grad{f}(y))$. Throughout this proof we consider the index range $k \geq 0$. From the global lower bound property of $w_k$ we have $R_{k + 1}(y) \geq 0$ for all $y \in \mathbb{E}$. It follows that $A_k R_{k + 1}(y_k) + a_{k + 1} R_{k + 1}(x^*) \geq 0$. Expanding this result using $\grad{f}(x^*) = 0$ and grouping terms yields
\begin{equation} \label{label_110}
\begin{gathered}
A_k f(y_k) + a_{k + 1} f(x^*) \geq A_{k + 1} f(y_{k + 1}) + \frac{A_k}{2 L_f} \| \grad{f}(y_k) \|^2 + \frac{A_{k + 1}}{2 L_f} \| \grad{f}(y_{k + 1}) \|^2\\ + \left\langle \grad{f}(y_{k + 1}), A_k \left(y_k - \frac{1}{L_f} B^{-1} \grad{f}(y_k) \right) + a_{k + 1} x^* - A_{k + 1} y_{k + 1} \right\rangle.
\end{gathered}
\end{equation}
Using the same derivations as in \eqref{label_086} we have
\begin{equation} \label{label_111}
\frac{1}{2} \| v_{k} - x^* \|^2 - \frac{1}{2} \| v_{k + 1} - x^* \|^2 = \langle \grad{f}(y_{k + 1}), a_{k + 1} v_{k} - a_{k + 1} x^* \rangle - \frac{a^2_{k + 1}}{2} \| \grad{f}(y_{k + 1}) \|_*^2.
\end{equation}
Adding together \eqref{label_110} and \eqref{label_111} and rearranging terms completes the proof.
\end{proof}
A sufficient condition for Theorem~\ref{label_108} consists of \eqref{label_070} and \eqref{label_088} together with
\begin{equation} \label{label_112}
L_f a_{k + 1}^2 \leq 2 A_k + 2 a_{k + 1}.
\end{equation}
Taking equality in \eqref{label_112} we obtain the original online Optimized Gradient Method.

\section{Offline mode}

In Theorem~\ref{label_079} can be refactored, by expanding $\Gamma_{k + 1}$, shifting $\frac{A_{k + 1}}{2 L_f} \| f'(y_{k + 1} \|$ to the right-hand side and renaming the altered quantities thus yielding an alternative sequences $\bar{a}_k$, $\bar{y}_k$ and $\bar{\psi}_k$ for all $k \geq 1$ satisfying the following relation.

\begin{corollary} \label{label_113}
For any algorithm that employs primal dual estimate functions updated as in \eqref{label_074}, the alternative estimate function $\bar{\psi}_{k + 1}$ updated as
\begin{equation} \label{label_114}
\bar{\psi}_{k + 1}(y, g) = \psi_k(y, g) + \bar{a}_{k + 1} w_{k + 1}(y, g), \quad y \in \mathbb{E}, \quad g \in \mathbb{E}^*, \quad k \geq 0,
\end{equation}
using the simple lower bounds given by \eqref{label_075} and $\bar{a}_{k + 1} > 0$, satisfies for all $\bar{y}_{k + 1} \in \mathbb{E}$
\begin{equation} \nonumber
\begin{gathered}
\bar{\psi}_{k + 1}^* - f(\bar{y}_{k + 1}) \geq \Gamma_k + \left( \frac{A_k}{2 L_f} - \frac{\bar{a}_{k + 1}^2}{2} \right) \| \grad{f}(\bar{y}_{k + 1}) \|_*^2 \\ + \left\langle \grad{f}(\bar{y}_{k + 1}), A_k \left(y_k - \frac{1}{L_f}B^{-1}\grad{f}(y_k)\right) + \bar{a}_{k + 1} v_k - (A_k + \bar{a}_{k + 1}) \bar{y}_{k + 1} \right\rangle, \quad k \geq 0.
\end{gathered}
\end{equation}
\end{corollary}
Corollary~\ref{label_113} can be made to hold for every algorithmic state by setting the updates of $\bar{a}_k$ and $\bar{y}_k$ to $L_f \bar{a}^2_{k + 1} = A_k$, $(A_k + \bar{a}_{k + 1}) \bar{y}_{k + 1} = A_k x_k + \bar{a}_{k + 1} v_k$ for all $k \geq 0$. The convergence guarantees for $\bar{y}_{k + 1}$ are given by
\begin{equation}
f(\bar{y}_{k + 1}) - f(x^*) \leq \frac{1}{A_k + \bar{a}_{k + 1}} \bar{\psi}^*_{k + 1} - f(x^*) \leq \frac{1}{2 (A_k + \bar{a}_{k + 1})} \| x_0 - x^* \|^2, \quad k \geq 0.
\end{equation}
The oracle complexities of computing $\bar{y}_{k + 1}$ and $x_k$ are identical, yet $\bar{y}_{k + 1}$ has a slightly larger convergence guarantee.

The results in Corollary~\ref{label_113} are compatible with augmentation and Theorem~\ref{label_108} can likewise be refactored to produce an auxiliary sequence with slightly large guarantees for the same oracle complexity.

\begin{corollary} \label{label_115}
Any first-order method that maintains the state variables $A_k$, $a_k$, $y_k$ and $v_k$ satisfies for any $\bar{a}_{k + 1} > 0$ and $\bar{y}_{k + 1} \in \mathbb{E}$ the following:
\begin{equation}
\begin{gathered}
\mathcal{D}_{k} \geq \bar{A}_{k + 1} (f(\bar{y}_{k + 1}) - f(x^*)) + \frac{1}{2} \|\bar{v}_{k + 1} - x^* \|^2 + \left( \frac{\bar{A}_{k + 1}}{2 L_f} - \frac{\bar{a}_{k + 1}^2}{2} \right) \| \grad{f}(\bar{y}_{k + 1}) \|^2 \\ + \left\langle \grad{f}(\bar{y}_{k + 1}), A_k \left(y_k - \frac{1}{L_f}B^{-1}\grad{f}(y_k)\right) + \bar{a}_{k + 1} v_k - \bar{A}_{k + 1} y_{k + 1} \right\rangle, \quad k \geq 0,
\end{gathered}
\end{equation}
where $\bar{A}_{k + 1} \defq A_k + \bar{a}_{k + 1}$ and $\bar{v}_{k + 1} \defq v_k - \bar{a}_{k + 1} \grad{f}(\bar{y}_{k + 1})$.
\end{corollary}
Corollary~\ref{label_115} can be easily upheld regardless of the algorithmic state if $a_{k + 1}$ and $\bar{y}_{k + 1}$ are obtained using a fixed step of the Fast Gradient Method, namely $L_f \bar{a}^2_{k + 1} = \bar{A}_{k + 1}$ and $\bar{A}_{k + 1} \bar{y}_{k + 1} = A_k x_k + \bar{a}_{k + 1} v_k$. The oracle complexity of computing $\bar{y}_{k + 1}$ is equal to that of $x_k$ in this case and, if all previous iterations are OGM iterations, by virtue of Theorem~\ref{label_108} we have that
\begin{equation}
f(\bar{y}_{k + 1}) - f(x^*) \leq \frac{1}{\bar{A}_{k + 1}} \mathcal{D}_k \leq \frac{1}{\bar{A}_{k + 1}} \mathcal{D}_0 = \frac{1}{2 (A_k + \bar{a}_{k + 1})} \| x_0 - x^* \|^2, \quad k \geq 0.
\end{equation}
In both our estimate sequence based algorithm and offline OGM, the above gains in computational complexity hold only in offline mode, i.e., when the total number of iterations is known in advance. In online mode, it is necessary to compute at every iteration $k \geq 0$ a stopping criterion involving $f(\bar{y}_{k + 1})$. This additional point is only used for this purpose, because we need to compute $y_{k + 1}$ to advance the algorithm further. The online algorithms can efficiently reuse the oracle information, as we shall elaborate in Section~\ref{label_116}. When running the algorithms for hundreds of iterations, the overhead of the offline mode eclipses the gains, amounting to no more than one iteration. For this reason, we will only consider the online methods in our simulations.

\section{Numerical results} \label{label_116}

To showcase the importance of our new bound and of the methods employing it, we have performed simulations on two particularly difficult optimization problems: a very ill-conditioned quadratic problem (QUAD) and a logistic regression problem with sparse design (LRSP). These problems notably lack additional structure, such as quadratic functional growth~\cite{ref_016}, that could be exploited by fixed point schemes.

The objective and gradient in QUAD are given by
\begin{equation}
f(x) = \frac{1}{2} \langle x, A x \rangle, \quad \grad{f}(x) = A x,
\end{equation}
respectively, where matrix $A$ is diagonal of size $n = 1000$ with the diagonal entries given by $\sigma_i = \sin^2\left(\frac{\pi i}{2 n}\right)$ for $i = 1, ... , n$. The starting point $x_0$ was selected as to have $(x_0)_i = 1 / \sigma_i$, $i \in \{1, ...,n\}$. Even though the matrix $A$ is sparse, we have resorted to a dense matrix implementation because an equivalent problem can be obtained by means of rotation, or more generally by applying an isometry, whereby the matrix representation becomes necessarily dense. Note that when such a transform is unknown to the algorithm, the problem of reversing it with the aim of recovering the original sparse matrix entails a far larger computational effort than the one of solving the original problem with a dense matrix implementation.

In LRSP the oracle functions are, respectively, given by
\begin{equation}
f(x) = \mathcal{I}(A x) - \langle y, A x \rangle, \quad \grad{f}(x) = A^T (\mathcal{L}(A x) - y),
\end{equation}
where $A$ is an $m = 10000$ by $n = 2000$ matrix, $y$ is a vector of size $m$. The sum softplus function $\mathcal{I}(x)$ and the element-wise logistic function $\mathcal{L}(x)$ are defined as
\begin{equation}
\mathcal{I}(x) = \displaystyle \sum_{i = 1}^{m} \log(1 + e^{-x_i}), \quad (\mathcal{L}(x))_i = \frac{1}{1 + e^{-x_i}}, \quad i \in \{1, ...,m\}.
\end{equation}
Here the matrix $A$ is sparse, with only $0.1\%$ of elements being non-zero, themselves sampled from the standard normal distribution. The values $y_i$ are sampled from $\{0, 1\}$ using the probability distribution $\mathbb{P}(Y_i = 1) = \mathcal{L}(A x_0)_i$, $i \in \{1, ...,m\}$. In this case we have resorted to a sparse implementation of the matrix $A$ to keep running times of the same magnitude as those in QUAD.

Also concerning oracle implementation, we have used the fact that $f(x)$ and $\grad{f}(x)$ share subexpressions on both problems. Thus, after a call to $\grad{f}(y_k)$ is made, the computational cost of $f(y_k)$ becomes negligible in QUAD and very low in LRSP.

The optimum point has the closed form expression $x^* = 0$ in QUAD. To obtain an estimate of the LRSP optimum we had to run Nesterov's original FGM for $10000$ iterations.

We have tested the same collection of algorithms on both problems. Our benchmark contains the original FGM, the online version of OGM (with all iterations using the same updates), the original GMM~\cite{ref_021}, our Improved Gradient Method with Memory (IGMM), the Accelerated Gradient Method with Memory (AGMM) equipped with line-search \cite{ref_006,ref_008} and our Optimized Gradient Method with Memory. The methods with memory are tested using a bundle size increasing exponentially from $1$ to $256$. Note that AGMM and OGMM do not allow a bundle size of $1$. In this case, AGMM designates line-search ACGM, as described in \cite{ref_009,ref_010}, and OGMM becomes our version of online OGM with the weight update given by equality in \eqref{label_089} leading to \eqref{label_090}.

The auxiliary process $\lambda$ is implemented in GMM using an optimized form of the Frank-Wolfe method~\cite{ref_011} as specified in~\cite{ref_021}. IGMM, AGMM and OGMM use a version of the Fast Gradient Method adapted to constrained problems with smooth objectives~\cite{ref_018}. All methods with memory update the bundle using a cyclic replacement strategy (CRS), whereby the new entries displace the oldest.

The termination criteria differ between algorithms. FGM, GMM, IGMM and ACGM are stopped when $f(x_k) < \Theta$, where $\Theta = f^* + \eabs$ is a threshold value that attains the absolute function value error $\eabs$. OGM and OGMM do not call the oracle at $x_k$ and are stopped instead when $f(y_k) - \frac{1}{2 L_f} \|\grad{f}(y_k)\|_*^2 < \Theta$. Throughout our experiments we have used the standard Euclidean norm with $B = I_n$, the identity matrix.

To keep the number of iterations of the fixed-point schemes within one million and the iterations of the fastest algorithms above the maximum tested bundle size, we have selected $\erel = 10^{-4}$ for QUAD and $\erel = 10^{-3}$ for LRSP, where the relative error is defined as $\erel = \eabs / (f(x_0) - f^*)$. On both problems, we have chosen for GMM a tolerance $\delta = \eabs / 2$, as recommended by \cite{ref_021}, whereas for all other methods with memory we have selected a much lower subprocedure target tolerance of $\delta = 10^{-3} \eabs$. For AGMM and OGMM employing Newton's method for convergence guarantee adjustment, we have limited the Newton iterations to $N = 2$ and the number of inner iterations of the subprocedure $\lambda$ to $10$, establishing a limit of $20$ inner iterations per outer iteration (see also \cite{ref_008}). In the case of GMM, no limit can be placed on the number of inner iteration per each call to $\lambda$ while for IGMM we have set the limit again to $20$. All methods equipped with line-search used the standard parameter choice $r_u = 2$ and $r_d = 0.5$.

When testing our methods with memory, we first consider the scenario in which the exact value of the global Lipschitz constant is known to the algorithms ($L = L_f$). Table~\ref{label_117} lists the number of iterations until termination for the gradient methods with memory on QUAD. Column $m$ designates the bundle sizes, Outer lists the number of outer iterations while Inner stands for average number of inner iterations per outer iteration. Table~\ref{label_118} lists the corresponding running times. Here Time (s) denotes the expended wall-clock time in seconds while IT (ms) shows the average running time per outer iteration measured in milliseconds. Iterations until convergence and the corresponding running times of the methods with memory when run on LRSP, also with $L = L_f$, are listed in Tables \ref{label_119} and \ref{label_120}, respectively.

\begin{table}[h!]
\caption{Iterations until the relative accuracy $\erel = 10^{-4}$ is reached on QUAD with $L = L_f$. Outer marks the number of outer iterations while Inner stands for the average number of inner iterations per outer iteration.}
\label{label_117}
\centering \footnotesize
\begin{tabular}{|r|rr|rr|rr|rr|} \hline
m & \tcol{GMM} & \tcol{IGMM} & \tcol{AGMM} & \tcol{OGMM} \\
& Outer & Inner & Outer & Inner & Outer & Inner & Outer & Inner \\ \hline
1 & 146862 & 0.00 & 146862 & 0.00 & 1441 & 0.00 & 1273 & 0.00 \\
2 & 218988 & 0.00 & 84206 & 0.56 & 1256 & 8.56 & 1241 & 7.24 \\
4 & 383851 & 0.01 & 77300 & 0.81 & 1092 & 17.11 & 930 & 14.11 \\
8 & 249834 & 1.23 & 79251 & 0.78 & 1093 & 17.41 & 936 & 15.68 \\
16 & 225952 & 1.94 & 82145 & 0.85 & 1094 & 18.41 & 938 & 17.97 \\
32 & 224279 & 2.22 & 86235 & 0.91 & 1092 & 19.28 & 934 & 19.23 \\
64 & 267367 & 1.27 & 92720 & 0.87 & 1086 & 19.64 & 929 & 19.92 \\
128 & 247196 & 0.85 & 96528 & 0.91 & 1079 & 19.93 & 921 & 19.98 \\
256 & 209815 & 0.95 & 97708 & 0.99 & 1063 & 19.96 & 906 & 19.98 \\ \hline
\end{tabular}

\vspace{3mm}

\captionof{table}{Running time until $\erel = 10^{-4}$ is reached on QUAD with $L = L_f$. Time marks the total running time in seconds and IT the per-iteration running time in milliseconds.}
\label{label_118}
\centering \footnotesize
\begin{tabular}{|r|rr|rr|rr|rr|} \hline
m & \tcol{GMM} & \tcol{IGMM} & \tcol{AGMM} & \tcol{OGMM} \\
& Time (s) & IT (ms) & Time (s) & IT (ms) & Time (s) & IT (ms) & Time (s) & IT (ms) \\ \hline
1 & 357.28 & 2.43 & 356.51 & 2.43 & 6.83 & 4.74 & 1.52 & 1.19 \\
2 & 563.84 & 2.57 & 202.69 & 2.41 & 5.96 & 4.74 & 1.48 & 1.20 \\
4 & 1037.89 & 2.70 & 186.53 & 2.41 & 5.19 & 4.75 & 1.12 & 1.20 \\
8 & 615.14 & 2.46 & 192.92 & 2.43 & 5.21 & 4.77 & 1.14 & 1.22 \\
16 & 570.37 & 2.52 & 203.76 & 2.48 & 5.25 & 4.80 & 1.17 & 1.25 \\
32 & 590.98 & 2.64 & 222.61 & 2.58 & 5.31 & 4.86 & 1.23 & 1.31 \\
64 & 769.78 & 2.88 & 258.42 & 2.79 & 5.42 & 4.99 & 1.35 & 1.46 \\
128 & 835.52 & 3.38 & 311.96 & 3.23 & 5.71 & 5.29 & 1.68 & 1.82 \\
256 & 923.00 & 4.40 & 410.85 & 4.20 & 6.19 & 5.82 & 2.43 & 2.68 \\ \hline
\end{tabular}
\end{table}

\begin{table}[h!]
\caption{Iterations until $\erel = 10^{-3}$ is reached on LRSP with $L = L_f$}
\label{label_119}
\centering \footnotesize
\begin{tabular}{|r|rr|rr|rr|rr|} \hline
m & \tcol{GMM} & \tcol{IGMM} & \tcol{AGMM} & \tcol{OGMM} \\
& Outer & Inner & Outer & Inner & Outer & Inner & Outer & Inner \\ \hline
1 & 25311 & 0.00 & 25311 & 0.00 & 480 & 0.00 & 505 & 0.00 \\
2 & 23623 & 0.01 & 9100 & 1.16 & 369 & 11.31 & 319 & 11.00 \\
4 & 44318 & 0.02 & 8011 & 1.39 & 322 & 18.54 & 313 & 16.12 \\
8 & 43901 & 0.76 & 7519 & 1.71 & 327 & 18.53 & 320 & 16.88 \\
16 & 34805 & 1.38 & 8330 & 1.75 & 323 & 18.74 & 327 & 17.15 \\
32 & 30349 & 1.71 & 8367 & 2.06 & 316 & 19.74 & 322 & 18.79 \\
64 & 31973 & 1.22 & 8640 & 2.18 & 288 & 19.81 & 310 & 19.79 \\
128 & 29373 & 0.94 & 8968 & 2.35 & 287 & 19.81 & 305 & 19.79 \\
256 & 25355 & 1.12 & 9542 & 2.69 & 286 & 19.81 & 305 & 19.79 \\ \hline
\end{tabular}

\vspace{3mm}

\captionof{table}{Running time until $\erel = 10^{-3}$ is reached on L1LR with $L = L_f$}
\label{label_120}
\centering \footnotesize
\begin{tabular}{|r|rr|rr|rr|rr|} \hline
m & \tcol{GMM} & \tcol{IGMM} & \tcol{AGMM} & \tcol{OGMM} \\
& Time (s) & IT (ms) & Time (s) & IT (ms) & Time (s) & IT (ms) & Time (s) & IT (ms) \\ \hline
1 & 27.13 & 1.07 & 26.90 & 1.06 & 1.01 & 2.10 & 0.31 & 0.61 \\
2 & 25.65 & 1.09 & 9.77 & 1.07 & 0.78 & 2.12 & 0.20 & 0.62 \\
4 & 49.33 & 1.11 & 8.75 & 1.09 & 0.68 & 2.12 & 0.20 & 0.63 \\
8 & 51.31 & 1.17 & 8.50 & 1.13 & 0.70 & 2.13 & 0.21 & 0.64 \\
16 & 44.56 & 1.28 & 10.09 & 1.21 & 0.71 & 2.19 & 0.22 & 0.69 \\
32 & 45.59 & 1.50 & 11.43 & 1.37 & 0.72 & 2.27 & 0.25 & 0.77 \\
64 & 62.26 & 1.95 & 14.41 & 1.67 & 0.70 & 2.43 & 0.28 & 0.91 \\
128 & 83.45 & 2.84 & 20.39 & 2.27 & 0.76 & 2.66 & 0.35 & 1.15 \\
256 & 120.35 & 4.75 & 33.65 & 3.53 & 0.82 & 2.87 & 0.42 & 1.38 \\ \hline
\end{tabular}
\end{table}

We see immediately that the fixed-point methods require a computational effort to converge at least an order of magnitude higher than the accelerated methods, measured either in iterations or running time. To keep running times within reasonable limits, we exclude GMM and IGMM from further analysis.

Next, to investigate how the convergence guarantee adjustment procedure can act as a line-search substitute, we test the accelerated methods with memory this time supplied with a value overestimated by a factor of $4$ ($L = 4 L_f$). The results on QUAD in both iterations and running times are listed in Table~\ref{label_121}. The corresponding values on LRSP can be found in Table~\ref{label_122}.

\begin{table}[h!]
\caption{Iterations and running time until $\erel = 10^{-4}$ is reached on QUAD with $L = 4 L_f$}
\label{label_121}
\centering \footnotesize
\begin{tabular}{|r|rr|rr|rr|rr|} \hline
m & \multicolumn{4}{c|}{AGMM} & \multicolumn{4}{c|}{OGMM} \\
& Outer & Inner & Time (s) & IT (ms)& Outer & Inner & Time (s) & IT (ms) \\ \hline
1 & 1441 & 0.00 & 6.83 & 4.74 & 2547 & 0.00 & 3.02 & 1.18 \\
2 & 1259 & 8.63 & 5.97 & 4.75 & 1451 & 9.14 & 1.74 & 1.20 \\
4 & 1093 & 17.11 & 5.19 & 4.75 & 1451 & 11.72 & 1.74 & 1.20 \\
8 & 1094 & 17.39 & 5.22 & 4.77 & 1454 & 12.48 & 1.77 & 1.22 \\
16 & 1095 & 18.57 & 5.25 & 4.80 & 1459 & 14.27 & 1.82 & 1.24 \\
32 & 1092 & 19.25 & 5.31 & 4.86 & 1462 & 16.85 & 1.91 & 1.31 \\
64 & 1086 & 19.69 & 5.42 & 4.99 & 1464 & 18.33 & 2.09 & 1.43 \\
128 & 1075 & 19.96 & 5.69 & 5.29 & 1459 & 19.92 & 2.49 & 1.71 \\
256 & 1068 & 19.96 & 6.22 & 5.83 & 1455 & 19.93 & 3.19 & 2.19 \\ \hline
\end{tabular}

\vspace{3mm}

\captionof{table}{Iterations and running time until $\erel = 10^{-3}$ is reached on LRSP with $L = 4 L_f$}
\label{label_122}
\centering \footnotesize
\begin{tabular}{|r|rr|rr|rr|rr|} \hline
m & \multicolumn{4}{c|}{AGMM} & \multicolumn{4}{c|}{OGMM} \\
& Outer & Inner & Time (s) & IT (ms)& Outer & Inner & Time (s) & IT (ms) \\ \hline
1 & 481 & 0.00 & 1.01 & 2.10 & 1010 & 0.00 & 0.61 & 0.61 \\
2 & 368 & 10.91 & 0.78 & 2.12 & 565 & 10.31 & 0.35 & 0.62 \\
4 & 324 & 18.52 & 0.69 & 2.13 & 573 & 12.02 & 0.36 & 0.63 \\
8 & 329 & 18.48 & 0.71 & 2.15 & 577 & 13.73 & 0.37 & 0.65 \\
16 & 324 & 18.71 & 0.71 & 2.19 & 584 & 16.33 & 0.40 & 0.69 \\
32 & 313 & 19.79 & 0.71 & 2.27 & 587 & 19.09 & 0.45 & 0.77 \\
64 & 293 & 19.78 & 0.71 & 2.42 & 584 & 19.95 & 0.54 & 0.92 \\
128 & 292 & 19.78 & 0.78 & 2.66 & 585 & 19.95 & 0.70 & 1.20 \\
256 & 292 & 19.78 & 0.84 & 2.86 & 586 & 19.95 & 0.95 & 1.62 \\ \hline
\end{tabular}
\end{table}

Finally, we list the results in iterations and running times on QUAD and LRSP with either $L = L_f$ or $L = 4 L_F$ for the methods without memory in Table~\ref{label_123}.

\begin{table}[h]
\caption{Performance of the original FGM and online OGM}
\label{label_123}
\centering \footnotesize
\begin{tabular}{|l|rrr|rrr|} \hline
Problem & \multicolumn{3}{c|}{FGM} & \multicolumn{3}{c|}{OGM} \\
& Outer & Time (s) & IT (ms) & Outer & Time (s) & IT (ms) \\ \hline
QUAD ($\erel = 10^{-4}$, $L = L_f$) & 1795 & 2.16 & 1.20 & 1269 & 1.52 & 1.19 \\
QUAD ($\erel = 10^{-4}$, $L = 4 L_f$) & 3596 & 4.27 & 1.19 & 2542 & 3.01 & 1.19 \\
LRSP ($\erel = 10^{-3}$, $L = L_f$) & 711 & 0.43 & 0.61 & 502 & 0.30 & 0.61 \\
LRSP ($\erel = 10^{-3}$, $L = 4 L_f$) & 1425 & 0.87 & 0.61 & 1007 & 0.61 & 0.61 \\ \hline
\end{tabular}
\end{table}

GMM performs very poorly on all tested instances, with the number of outer iterations actually increasing with bundle size. This is consistent with the previous findings in \cite{ref_007}, suggesting that the CRS generally impedes performance on difficult problems. Interestingly, even though the mode of operation in IGMM is very similar to GMM, our bound manages to improve performance with CRS. Even so, unlike the behavior seen in \cite{ref_021}, increasing the bundle size beyond $4$ does not result in further improvements. While generally benefiting from the bundle, all tested fixed-point schemes converge very slowly even when an accurate value of $L_f$ is supplied so we turn our attention to the accelerated methods.

We first consider the non-adaptive schemes. On both difficult problems, online OGM manages to surpass fixed-step FGM in every instance and, as predicted theoretically, requires fewer outer iterations than OGMM with $m = 1$, although the difference is never greater than $0.6 \%$. This discrepancy is not even discernible when the convergence speed is measured in wall-clock time. Providing an inaccurate estimate of $L_f$ impacts the performance of FGM and OGM according to the worst-case bounds: a four-fold overestimate approximately doubles (the ratio actually ranges between $2.003$ and $2.006$) the number of outer iterations.

When $L = L_f$, both accelerated methods with memory manage to reduce the number of outer iterations by around a quarter for a bundle of size $m = 4$ on both QUAD and LRSP. Increasing the bundle beyond this level results only in a moderate decrease in outer iterations at the expense of increased running times. When $m = 4$, OGMM converges slightly faster in iterations than AGMM. The running times of OGMM are less than a third of those of AGMM, the discrepancy stemming from the per outer iteration running times. This difference arises from the different adaptive mechanisms employed by the two methods. Our realistic oracle implementation allows the simultaneous computation of function value and gradient \emph{at the same point}, with the additional cost incurred by the function value being negligible. OGMM calls the oracle only at the points $y_k$, using the pair $(f(y_k), \grad{f}(y_k))$ to update the model and adjust the convergence guarantees without the need for additional oracle calls, whereas AGMM has a line-search procedure that also requires the computation of $f(x_k)$. The backtracking line-search of AGMM has a high failure rate with an average of one failure per outer iteration. Every failure entails an additional call to the oracle at \emph{both} $y_k$ and $x_k$, explaining the high per-iteration cost of AGMM.

When the methods are supplied with an inaccurate estimate $L = 4 L_f$, the performance of AGMM is virtually unaffected. This is to be expected considering the fully adaptive nature of its line-search procedure. The improved bound employed by OGMM relies on the value of $L_f$ being known, and the convergence guarantee adjustment procedure cannot fully compensate for this drawback. Consequently OGMM lags behind in iterations. The lag in LRSP is much larger than in QUAD, but this is explained by AGMM being able to utilize the bundle more efficiently on this problem. OGMM still manages to lead in running times because it calls the oracle in a single point at every outer iteration. AGMM gained the most using $m = 4$ but for OGMM the optimal bundle size was the very small $m = 2$. This is most likely due to the bound employed by OGMM being inaccurate and interfering with larger bundles.

For moderately sized bundles up to $m = 32$, the overhead for every method with memory was negligible in terms of the impact on per outer iteration running times on every problem instance tested.

\section{Discussion} \label{label_124}

In this work we have constructed an optimal lower bound for smooth convex objectives based solely on the information available to a black-box first-order scheme at any point: the global properties of the function including the global Lipschitz constant $L_f$ along with the oracle output at a collection of points. From the perspective of the algorithm, the bound is indistinguishable from the original objective, thus constituting a perfect interpolation of collected information.

The bound does not have a closed form, instead being an optimized function, where the objective is a point-wise maximum of primal-dual bounds that introduce the additional variable $g$. We have provided two additional equivalent forms: a dual form in which the bound is a point-wise maximum of a quadratic function parameterized by $\lambda$ with the standard simplex and a geometric interpretation wherein the bound is the lower envelope of all the simple parabolic functions dominating the piece-wise linear lower model.

Optimization algorithms employing the bound may not be able to store the entire oracle history in memory and we have investigated how to employ bounds based on the aggregation of past information. We have studied both the impact of linearly aggregating the oracle information itself as well as the aggregation of the primal-dual bounds themselves. The resulting bounds remain valid, albeit weaker, and algorithms need to balance memory capacity with bound accuracy.

We have constructed a fixed-point scheme, the Improved Gradient Method with Memory, wherein our bound constructed around an aggregate of the oracle history replaces directly the piece-wise linear bound employed in the original Gradient Method with Memory. The tighter bound did not improve the worst-case guarantees but preliminary computational results show a marked performance increase when compared to the original GMM when $L_f$ is known.

However, it is accelerated schemes that appear to utilize our new bound to its full potential. First of all, the presence of the additional variable $g$ in the primal-dual bounds allow us to introduce the closely related primal-dual estimate functions. Slightly altering the design pattern for first-order accelerated methods in \cite{ref_010,ref_005} to accommodate the additional variable and utilizing a stricter estimate sequence property based on the descent rule \eqref{label_071}, we were able to construct a method that has the same worst-case rate as the Optimized Gradient Method, the fastest currently known, even up to the proportionality constant. Compared to OGM, our method has a very important additional feature: the estimate function optimal value is maintained at every iteration. This allowed us to incorporate the convergence guarantee adjustment procedure that is a able to mitigate to a good degree the lack of a line-search procedure. Moreover, the adjustment is entirely free of oracle calls.

The primal-dual estimate functions also allow us to add memory to our method. Whereas OGM relies only on a single vector aggregate of the past oracle history, we can store an oracle history subset of arbitrary size. Simulation results show a positive correlation between bundle size and convergence speed measured in outer iterations when the correct value of the global Lipschitz constant $L_f$ is fed to the algorithm. However, the gains beyond a bundle size of $m = 4$ are very small and the overhead mitigates the performance gains when measured using wall-clock running times.

Interestingly, our method needs exactly one combined gradient/function value oracle call per iteration. When the two functions are computed concurrently, our simulations show that our method is competitive even with methods with memory endowed with line-search. Our methods remain competitive also when $L_f$ is overestimated.

Our theoretical framework cannot just be used to construct a new algorithm. Employing primal augmentation we are able to recover the original online OGM and the potential functions used in \cite{ref_002} to study its convergence. Our design procedure and analysis relies solely on basic principles. All update rules stem from a simple and intuitive estimate sequence framework without the use of computerized assistance tools as in the Performance Estimate Framework.

An open question remains: whether the estimate sequence can be used to derive the original online OGM without augmentation and whether OGM itself can be endowed with an adaptive mechanism and memory.

\end{document}